\documentclass[11pt]{amsart}

\usepackage{epsfig}
\usepackage{graphicx}
\usepackage{amscd}
\usepackage{amsmath}
\usepackage{amsfonts}
\usepackage{mathrsfs}
\usepackage{amssymb}
\usepackage{verbatim}
\usepackage{color}
\usepackage{hyperref}

\newtheorem{theorem}{Theorem}[section]
\theoremstyle{plain}

\newtheorem{corollary}[theorem]{Corollary}

\newtheorem{defi}[theorem]{Definition}
\newtheorem{example}[theorem]{Example}

\newtheorem{lemma}[theorem]{Lemma}

\newtheorem{prop}[theorem]{Proposition}
\newtheorem{remark}[theorem]{Remark}

\numberwithin{equation}{section}

\def\what{\widehat}

\def\Jk{{\mathcal J}}
\def\Xxi{{\mathfrak X}_\zeta}

\def\FrE{{\mathfrak E}}

\def\FrT{{\mathfrak T}}

\def\fra{{\mathfrak f}}
\def\gra{{\mathfrak g}}

\def\Int{{\rm int}}
\def\freq{{\rm freq}}
\def\diam{{\rm diam}}

\def\supp{{\rm supp}}

\def\Ok{{\mathcal O}}

\newcommand{\lam}{\lambda}
\def\Lam{\Lambda}
\newcommand{\gam}{\gamma}
\newcommand{\om}{\omega}

\newcommand{\Gam}{\Gamma}
\newcommand{\sig}{\sigma}
\newcommand{\by}{{\bf y}}

\def\bbe{{\bf e}}
\newcommand{\R}{{\mathbb R}}
\newcommand{\Q}{{\mathbb Q}}
\newcommand{\Z}{{\mathbb Z}}
\newcommand{\C}{{\mathbb C}}

\def\N{{\mathbb N}}

\newcommand{\Prob}{{\mathbb P}\,}
\def\P{\Prob}

\def\ba{{\bf a}}
\def\wt{\widetilde}

\def\Rk{{\mathcal R}}

\def\Sk{{\mathcal S}}
\def\Lk{{\mathcal L}}
\def\Mk{{\mathcal M}}

\def\Ok{{\mathcal O}}

\def\Pk{{\mathcal P}}

\def\Dk{{\mathcal D}}
\def\Sf{{\sf S}}
\def\T{{\mathbb T}}

\def\bx{{\mathbf x}}

\def\X{X}
\def\bz{{\mathbf z}}

\def\be{\begin{equation}}
\def\ee{\end{equation}}
\newcommand{\Ek}{{\mathcal E}}

\newcommand{\eps}{{\varepsilon}}
\newcommand{\es}{\emptyset}

\def\ov{\overline}

\newcommand{\const}{{\rm const}}

\def\Cc{{\mathscr M}}
\def\Ac{{\mathscr A}}
\def\Cscr{{\mathscr C}}
\def\Mc{{\mathscr M}}

\def\Bc{{\mathscr B}}

\def\Sc{{\mathscr S}}
\def\Tc{{\mathscr T}}
\def\Lc{{\mathscr L}}

\def\doh{{\sf d}}

\def\ve1{\vec{1}}

\def\Tk{{\mathcal T}}
\def\Ak{{\mathcal A}}

\def\beps{{\boldsymbol \epsilon}}
\def\Kb{{\mathbf K}}

\def\bt{{\bf t}}
\def\bv{{\bf v}}

\def\clos{{\rm clos}}

\def\lamb{\boldsymbol \lam}
\def\Psib{\boldsymbol \Psi}

\begin{document}

\author{Boris Solomyak}
\address{Boris Solomyak\\ Department of Mathematics,
Bar-Ilan University, Ramat-Gan, Israel}

\thanks{The research of B. S. was supported by the Israel Science Foundation grant 911/19.}

\email{bsolom3@gmail.com}
\author{Rodrigo Trevi\~no}
\address{Rodrigo Trevi\~no \\ Department of Mathematics \\ University of Maryland, College Park}
\thanks{}
\email{rodrigo@trevino.cat}

\date{\today}

\title{Spectral cocycle for substitution tilings}

\begin{abstract}
The construction of spectral cocycle from the case of 1-dimensional substitution flows  \cite{BuSo20a} is extended to the setting of pseudo-self-similar tilings in $\R^d$,
allowing expanding similarities with rotations. The pointwise upper
 Lyapunov exponent of this cocycle is used to bound the local dimension of spectral measures of deformed tilings.
The deformations are considered, following \cite{Tre20}, in the simpler, non-random setting. We review some of the results on quantitative weak mixing from \cite{Tre20} in this special case
and illustrate them on concrete examples.
\end{abstract}

\maketitle

\begin{flushright}
{ \large \em Dedicated to the memory of Uwe Grimm}
\end{flushright}

\bigskip

\section{Introduction}

\thispagestyle{empty}

We extend the construction of spectral cocycle and partially extend the results, from the case of 1-dimensional substitution flows treated in \cite{BuSo20a}, to higher dimensions.
Another motivation for us is to make the results of \cite{Tre20} more accessible by presenting them in the simplest possible nontrivial case of a single self-similar tiling (corresponding to a stationary Bratteli diagram in \cite{Tre20}). We also indicate how the results of \cite{Tre20}, where it is assumed that the expansion map is a pure dilation, may be extended
to the case of general expanding similarities with rotations, which necessitates dealing with fractal boundaries and passing from self-similar to pseudo-self-similar tiling spaces.
We do not repeat the rather technical proofs of \cite{Tre20}, but illustrate the results on concrete examples (Kenyon's tilings defined via free group endomorphisms \cite{Kenyon} and a ``square'' tiling). 

The {\em spectral cocycle} was introduced in \cite{BuSo20a} for translation flows and $S$-adic systems; here we are concerned with the case of a single substitution. Briefly, given a 
primitive aperiodic substitution $\zeta$ on $m$ symbols with a non-singular substitution matrix $\Sf_\zeta$, the spectral cocycle from \cite{BuSo20a} is a complex matrix $m\times m$ cocycle over the endomorphism of the $m$-torus induced by the transpose $\Sf_\zeta^{\sf T}$. 
{
(In fact, in the case when $\Sf_\zeta$ is singular, one can restrict the underlying system to a lower-dimensional sub-torus on which non-singularity holds, see \cite{Rotem}.) }
Implicitly, the spectral cocycle appeared already in \cite{BuSo14} as a generalized matrix Riesz product.
In \cite{BuSo20a} it was shown that the (top, pointwise upper) Lyapunov exponents of this cocycle in some sense control the local spectral behavior of substitution $\R$-actions -- suspension flows over the usual substitution $\Z$-action, with a piecewise-constant roof function. By ``control'' we mean lower and upper bounds for the local dimension of spectral measures of cylindrical functions, see \cite{BuSo14}  for details. Independently, Baake et al \cite{BFGR,BGM} introduced a {\em Fourier matrix cocycle}, 
which coincides with the spectral cocycle on the one-dimensional $\Sf_\zeta^{\sf T}$-invariant manifold on the torus corresponding to the Perron-Frobenius eigenvector; it is tailored for the spectral analysis of (geometrically) self-similar tiling flows and was used to prove singularity of the diffraction spectrum for a class of non-Pisot substitution systems. { In \cite{Tre20}, it appeared as the application of traces applied to elements of certain AF algebras.}

\thispagestyle{empty}

The tools of spectral cocycle (without calling it such) provided a framework for the proof of almost sure H{\"o}lder regularity of translation flows on higher genus flat surfaces, first in \cite{BuSo18} for genus 2 and then in \cite{BuSo21} for an arbitrary genus greater than 1, { including many surfaces of infinite genus and finite area}. (The last paper appeared after the preprint of Forni \cite{Forni19} who used a different technique.)

The work \cite{Tre20} extended almost sure H\"older regularity results of \cite{BuSo21} to the case of (globally) random substitution tilings in $\R^d$. In a sense, this is a higher-dimensional
version of an $S$-adic system, in which self-similar tile substitutions are applied randomly, according to an underlying ergodic dynamical system, resulting in a tiling $\R^d$-action. A crucial distinction with the one-dimensional case is that one needs to define what is meant by a ``deformation'' of the tiling dynamical system. Whereas for $d=1$ it was rather natural to vary the tile lengths, or equivalently, view the system as a suspension with a piecewise-constant roof function and vary the ``heights'', this issue becomes more complicated for $d\ge 2$. In this setting {\em admissible deformations} were studied,
first by Clark and Sadun \cite{CS06} and then by Kellendonk \cite{Kel08} and Julien and Sadun \cite{JS}; this required dealing with the (e.g., \v{C}ech) cohomology group $\check H^1(X,\R^d)$, where $X$ is the tiling space. The spectral cocycle for $d\ge 2$
is still a complex matrix $m\times m$ cocycle, where $m$ is the number of prototiles, but the  cocycle is over the endomorphism induced by the inflation map on the {first} cohomology group. Incidentally, this point of view offers an advantage even for $d=1$.

The Fourier matrix cocycle has been introduced for self-similar tilings in \cite{BGM}, see also \cite{Baake3}; it is a cocycle over the inflation map of the entire $\R^d$ and it coincides with the spectral cocycle on a measure zero subset. In this case it was also used to prove singularity for some systems, such as the Godr\`{e}che-Lan\c{c}on-Billard tiling.


The paper is organized as follows. In Section~\ref{sec:back} we recall the background on tilings and tiling dynamical systems, which is mostly standard and appeared in many articles and books.
We also discuss pseudo-self-similar tilings, which are needed in order to define deformations when the prototiles of the original self-similar tiling have fractal boundaries.
Section~\ref{sec:coh} deals with cohomology and deformations. It also includes a treatment of geometric properties of deformed tiling spaces and their consequences, in particular, unique ergodicity and a
formula for measures of cylinder sets. In Section~\ref{sec:coc} we define the spectral cocycle and state the main result on local dimension estimates. In Section~\ref{sec:eig} we first state the result saying that for deformed tiling spaces weak mixing is equivalent to topological weak mixing, which was characterized in \cite{CS06}.
 The proof is a minor variation of the argument from \cite{SolTil}; we sketch it in Section~\ref{sec:proofs}.  Next we explain how a natural quantitative strengthening of the condition, which we call the
 ``quantitative Host-Veech criterion'', yields H\"older regularity of spectral measures and quantitative weak mixing. 
 Further, we state a theorem on quantitative weak mixing, which is essentially a special case of \cite[Theorem 1]{Tre20}, 
but we believe it is more accessible and easier to apply. A brief proof outline is included. Finally, we revisit the case of one-dimensional substitution tilings and show how to extend
\cite[Th.\,4.1]{BuSo14} to the case of reducible substitution matrix.
Section~\ref{sec:ex} is devoted to examples. There is a detailed discussion of a family of pseudo self-similar planar tilings due to Kenyon \cite{Kenyon}, to which our results apply. An example with
of ``square tilings'' is included as well. 
Section~\ref{sec:proofs} contains the remaining proofs. 


\section{Background} \label{sec:back}

\subsection{Tilings and tiling spaces} A {\em tiling} $\Tk$ of $\R^d$ is a covering by compact sets, called {\em tiles}, such that their interiors don't intersect. { The tiles need not be homeomorphic to balls or even connected.}
We assume that every tile is a closure of its interior.
We will consider the {\em translation $\R^d$-action} on tiling spaces (defined below in \ref{eqn:space}): the translation of $\Tk$ by $t\in \R^d$ is $\{T-t:\ T\in \Tk\}$.
Strictly speaking, a tile is a pair $T=(A,j)$, where $A$ is the {\em support} of $T$ and $j$ is its {\em type}, {\em color}, or {\em label}. This is needed, since frequently we need to distinguish geometrically equivalent
tiles. However, to avoid cluttered notation, we will often think of tiles as sets with colors, or labels. The colors are assumed to be preserved under translations, linear maps, etc.

A {\em patch} of the tiling $\Tk$ is a finite subset of $\Tk$. For two patches $\Pk,\Pk'\subset \Tk$ we write $\Pk\sim \Pk'$ if there exists $t\in \R^d$ such that $\Pk' = \Pk-t$; such patches are called
{\em translation-equivalent}.
For a bounded $A\subset \R^d$ define the patches:
$$
\Ok^-_\Tk(A) \ = \ \mbox{the largest patch of $\Tk$ completely contained in $A$};
$$
$$
\Ok^+_\Tk(A) \ = \ \mbox{the smallest patch of $\Tk$ containing $A$}.
$$
We will always assume that all our tilings have (translational) {\em finite local complexity}, or FLC: for any $R>0$ there are finitely many patches of the form $\Ok_\Tk^-(B_R(x))$ up to translation equivalence. In particular, there are finitely many tiles up to translation. We fix a collection of representatives $\Ak = \{T_1,\ldots,T_m\}$ and call them {\em prototiles}. Another standing assumption will be that $\Tk$ is {\em repetitive}, i.e., for any patch $\Pk\subset \Tk$ there exists $R_\Pk>0$ such that any ball $B_{R_\Pk}(x)$, with $x\in \R^d$, contains a $\Tk$-patch $\Pk'\sim \Pk$.

The {\em tiling space} (or ``hull'') determined by $\Tk$ is the closure of the translation orbit:
\begin{equation}
  \label{eqn:space}
X_\Tk:= \clos\{\Tk-x:\ x\in \R^d\}
\end{equation}
in the standard ``local'' metric, in which $\Tk$ and $\Tk'$ are $\eps$-close for $\eps>0$ sufficiently small if and only if
 there exists $t\in B_\eps(0)$ such that $\Tk$ and $\Tk' - t$ coincide in $B_{1/\eps}(0)$.
The {\em tiling dynamical system} $(X_\Tk,\R^d)$ is the $\R^d$-action by translations.
If $\Tk$ is repetitive, then $X_\Tk = X_{\Tk'}$ for any $\Tk' \in X_{\Tk}$, so the tiling dynamical system is minimal. We will often omit the subscript $\Tk$ and simply write $X$ for a tiling space.
An alternative way to define the tiling space is via the {\em atlas} of admissible patches (analogous to a language in symbolic dynamics).

\subsubsection{Functions} Let $\Tk$ be a tiling of $\R^d$. A function $h:\X_\Tk\to \C$ is {\em transversally locally constant} (TLC) if there exists $R>0$ such that 
$$
\Ok^-_{\Tk'}(B_R(0)) = \Ok^-_{\Tk''}(B_R(0)) \implies h(\Tk') = h(\Tk'').
$$
A function $f:\R^d\to \C$ is $\Tk$-equivariant if there exists $R>0$ such that
\begin{equation}
  \label{PEfunction}
  \Ok^-_{\Tk}(B_R(x)) = \Ok^-_{\Tk}(B_R(y)) + (x-y) \implies f(x) = f(y).
  \end{equation}
Given a TLC-function $h:\X_\Tk\to \C$, we get a $\Tk$-equivariant function via $f(x) := h(\Tk-x)$.

\subsubsection{Mutual local derivability}

              { For two tilings $\Tk_1$ and $\Tk_2$, we say that $\Tk_2$ is {\em locally derivable (LD)} from $\Tk_1$ with radius $R>0$ if for all $x,y\in\mathbb{R}^d$,
                $$\Ok^-_{\Tk_1}(B_R(x)) = \Ok^-_{\Tk_1}(B_R(y))+(x-y)\Rightarrow \Ok^+_{\Tk_2}(\{x\}) = \Ok^+_{\Tk_2}(\{y\}) + (x-y).$$
                If $\Tk$ is LD from $\Sk$ and vice-versa, then 
              }
we say that tilings $\Tk$ and $\Sk$ are {\em mutually locally derivable} (MLD). 
 Note that this is similar in spirit to $\Tk$-equivariance. MLD implies conjugacy of tiling dynamical systems, but not the other way round, see \cite{Pet,RaSa}.
We will say that two minimal tiling spaces $X$ and $X'$ are MLD if there exist $\Tk\in X$ and $\Tk'\in X'$ that are MLD.

\subsubsection{Frequency of patches and unique ergodicity} Denote $Q_R=[-R,R]^d$.  The following is well-known and may be considered folklore. We refer the reader to \cite{LMS}, which is written in the language of Delone multisets; passing to the tiling setting is routine.

Given an FLC tiling $\Tk$, we say that $\Tk$ has {\em uniform patch frequencies} (UPF) if for any patch $\Pk\subset \Tk$, the limit
$$
\freq(\Pk,\Tk) = \lim_{n\to \infty} \frac{\#\{t\in \R^d: \ t+\Pk \subset Q_R + x\}}{(2R)^d}\ge 0
$$
exists uniformly in $x\in \R^d$. For a repetitive tiling space, the UPF property and the uniform frequencies do not depend on the choice of the tiling.

\begin{theorem}[{\cite[Theorem 2.7]{LMS}}]
Let $\Tk$ be an FLC tiling in $\R^d$. The dynamical system $(X_\Tk,\R^d)$ is uniquely ergodic if and only if $\Tk$ has the UPF property.
\end{theorem}

Under the UPF assumption, we also have an explicit formula for the measure of ``cylinder sets''. For a tiling space $X$ we define the {\em transversal} of a patch $\Pk$ by
$$
\Upsilon(\Pk) = \{\Tk \in X:\ \Pk\subset \Tk\}.
$$
For a patch $\Pk$ and a Borel set $U\subset \R^d$, the corresponding cylinder set is 
$
\Upsilon(\Pk)+U:= \{\Tk + x:\ \Tk\in \Upsilon(\Pk),\ x\in U\}.
$

\begin{prop}[{\cite[Corollary 2.8]{LMS}}] \label{prop-meas1}
Let $X$ be a uniquely ergodic FLC tiling space, with the unique invariant Borel probability measure $\mu$.
Let $\eta=\eta(X)$ be the diameter of the largest ball contained in every prototile. Then for any patch $\Pk$ and a Borel set $U$ with $\diam(U) < \eta(X)$, 
\be \label{eq-meas1}
\mu(\Upsilon(\Pk)+U) = \Lk^d(U)\cdot \freq(\Pk,\Tk).
\ee
\end{prop}
 

\subsection{Self-similar tilings} \label{sec:sst}
Now we define a self-similar tiling space, in the sense of \cite{Thur,Kenyon,SolTil}. 
Here we start with a substitution rule.
As before, suppose that we have a finite prototile set $\{\FrT_1,\ldots,\FrT_m\}$, where each $\FrT_j\subset \R^d$ is the closure of its interior, possibly with a fractal boundary. 
Let $\varphi$ be an expanding similarity on $\R^d$, with expansion constant $\theta>1$ (we do not assume that $\varphi$ is a pure dilation; in general, $\varphi = \theta\Ok$, where $\Ok$  is an orthogonal transformation). Assume that there is a {\em tile substitution}
\be \label{subs1}
\omega(\FrT_j) = \bigcup_{k\le m} (\FrT_k + \Dk_{jk}),\ \ j\le m,
\ee
where $\Dk_{jk}$ is a finite set of translations and the right-hand side represents a patch such that
\be \label{subs2}
\varphi A_j = \bigcup_{k} (A_k + \Dk_{jk}),\ \ j\le m,\ \ A_j = \supp(\FrT_j).
\ee
For a translated prototile the substitution acts by
\be \label{subs20}
\omega(\FrT_j + x) = \omega(\FrT_j) + \varphi(x),\ \  j\le m,\ x\in \R^d.
\ee
This and the property \eqref{subs2} imply that the substitution map can be iterated, resulting in larger and larger patches. In particular,
\be \label{subs3}
\omega^2(\FrT_j) = \{\FrT_s + \varphi \Dk_{jk} + \Dk_{ks}\}_{s\le m, k\le m},\ \ 1 \le j \le m,
\ee
etc. 
The {\em substitution tiling space} $X_\om$ corresponding to $\omega$ is, by definition, the collection of all tilings $\Tc$ of $\R^d$ such that every $\Tc$-patch is
 a sub-patch of $\omega^n(\FrT_j)$ for some $j\le m$ and $n\in \N$. We will sometimes omit the subscript $\om$ and simply write $X$ for the tiling space.
By definition, one can pass from $\om$ to $\om^k$ for $k\ge 2$, without affecting the substitution tiling space.
The substitution naturally extends to a continuous self-map of the tiling space $\om:\,X_\om\to X_\om$. This map is surjective, but not necessarily injective. Note that
\be \label{sub-tran}
\om(\Tc-t) = \om(\Tc) - \varphi(t),\ \ \Tc\in X_\om,\ t\in \R^d.
\ee
The {\em substitution matrix}  is defined by
\be \label{sub-matr}
\Sf_\om(j,k) = \#\Dk_{kj}.
\ee
We will assume that the substitution is {\em primitive}, i.e., some power of $\Sf_\om$ has all entries strictly positive. 
Repetitivity of $\Tc \in X_\om$ implies primitivity. A tiling is called {\em self-similar} if it is a fixed point of
the substitution: $\om(\Tc) = \Tc$. Such a tiling always exists, after passing to a higher power of $\om$, if necessary.
FLC repetitive  self-similar tiling spaces are known to be uniquely ergodic, see \cite{SolTil,LMS}.

\subsubsection{Aperiodicity and recognizability; hierarchical structure} We will assume that our tiling spaces are {\em aperiodic}, i.e., $\Tc = \Tc-t$ implies $t=0$. 
It follows from \eqref{sub-tran} that aperiodicity of $X_\om$ is necessary for $\om$ to be injective. It turns out that it is also sufficient \cite{SolUCP}. Then $\om$ is a homeomorphism, and we can ``desubstitute'' any $\Tc\in X_\om$ in a unique way; this property is usually called ``recognizability''. Equivalently, there is a unique way to compose the tiles of $\Tc$ into a collection of patches
of the form $\om(\FrT_j),\ j\le m$, in such a way that we get a tiling from the space $\varphi(X_\om)$.  These patches are called {\em supertiles} of order 1. This procedure can be iterated; thus any $\FrT\in \Tc$ is contained in a unique increasing sequence of supertiles of order $n$.

\subsubsection{Substitution Delone set associated with a self-similar tiling}

Let $\Tc$ be a self-similar tiling, $\om(\Tc) = \Tc$, which we assume exists, without loss of generality. Specify the location of each prototile $\FrT_j$ in $\Tc$. Then
$\om(\FrT_j)$ is a $\Tc$-patch, with the support $\varphi(\supp(\FrT_j))$. The definition \eqref{subs1} implies that each element of $\Dk_{jk}$ is translation vector between two
occurrences of equivalent tiles in $\Tc$, namely, $\FrT_k$ and its translate. This means that $\Dk_{jk}$ is a set of {\em return vectors}, which will be important in what follows.
We can write 
\be \label{dual1}
\Tc = \bigcup_{k=1}^m \bigl(\FrT_k + \Lc_k(\Tc)\bigr),
\ee
where $\Lc_k(\Tc)$ represents the set of locations of tiles of type $k$ in $\Tc$ (relative to the prototiles). Each $\Lc_k(\Tc)$ is a {\em Delone set},
that is, a uniformly discrete relatively dense set in $\R^d$. By convention, $0\in \Lc_k(\Tc)$ for each $k$.
Note that \eqref{dual1} and \eqref{subs1} yield a ``dual system of equations'' for the Delone sets:
\be \label{dual3}
\Lc_k(\Tc) = \coprod_{j\le m} \bigl(\varphi\Lc_j(\Tc) + \Dk_{jk}\bigr), \ \ k\le m.
\ee
This means that $\bigl(\Lc_k(\Tc)\bigr)_{k\le m}$ is a {\em substitution Delone multiset} in the sense of \cite{LW,LMS}.

\begin{defi} \label{def-SDel} \cite{LMS}.
A family of Delone sets $\bigl(\Lc_k\bigr)_{k\le m}$ is called a {\em substitution Delone multiset} with expansion $\varphi$ if 
there exist finite sets $\Dk_{jk}$ such that
\be \label{eq-SDel}
\Lc_k = \coprod_{j\le m} \bigl(\varphi\Lc_j + \Dk_{jk}\bigr), \ \ k\le m.
\ee
\end{defi}
This notion was introduced by Lagarias and Wang \cite{LW}, except that they allowed each $\Lam_k$ to be a set ``with multiplicities''.

\medskip

{

We will also need a set of {\em control points} for the tiling $\Tc$ and other tilings in $X_\om$. The sets $\Lc_k$ are not convenient, since they represent locations of tiles of type $k$
in $\Tc$ {\em relative} to the prototiles $\FrT_k\in \Tc$. To this end, pick a point $c(\FrT_k)\in \Int(\FrT_k)$ for each prototile, 
and let $c(\FrT_k+x) = c(\FrT_k)+x$ for translated tiles. Then
$$
\Lam_k(\Tc) := \Lc_k(\Tc) + c(\FrT_k)
$$
represents the set of control points in all the tiles of type $k$, and 
\be \label{subs101}
\Lam(\Tc) := \coprod_{k\le m} \Lam_k(\Tc)
\ee
 is a Delone set of all the control points for the tiling. Note that
$(\Lam_k(\Tc))_{k\le m}$ is a substitution Delone multiset as well, satisfying \eqref{eq-SDel}, with $\Dk_{jk}$ replaced by $\Dk_{jk} + c(\FrT_k) - \varphi c(\FrT_j)$.

}


\subsection{Pseudo self-similar (PSS) tilings} \label{sec:pss}

Our main object in this paper is {\em tiling deformations}, which are more conveniently defined for tilings whose tiles are convex polytopes meeting face-to-face.
When the tile boundaries are fractal, which is necessarily the case if the expansion map $\varphi$ involves an irrational rotation, an extra step is needed.
We are going to use {(a variant of)} the construction of {\em Derived Voronoi (DV) tilings} introduced by N. Priebe Frank \cite{NPF}, which turns a self-similar tiling into a pseudo-self-similar one, but with ``nice'' convex polytope tiles.

\begin{defi} {
Let $\varphi:\mathbb{R}^d\rightarrow \mathbb{R}^d$ be an {expanding similarity map}. A repetitive tiling $\mathcal{T}$ of finite local complexity is called a \emph{pseudo-self-similar tiling}, or PSS, with expansion $\varphi$ if $\mathcal{T}$ is locally derivable from $\varphi\mathcal{T}$.}
\end{defi}

{In fact, in \cite{NPF} the map $\varphi$ is allowed to be any expanding linear map, but here we restrict ourselves to similitudes.}

Let $\Lam$ be a Delone set in $\R^d$. The {\em Voronoi cell} of a point $x\in \Lam$ is, by definition,
$$
V(x) = \bigl\{\bt\in \R^d:\ \|\bt-x\| \le \|\bt-y\|,\ \forall\,y\in \Lam\setminus \{x\}\bigr\},
$$
and the corresponding {\em Voronoi tiling (tesselation)} is
$$
\Tk_\Lam:= \{V(x):\ x\in \Lam\}.
$$
The tiles of $\Tk_\Lam$ are convex polytopes meeting face-to-face. Applying this procedure to $\Lam(\Tc)$, we want to make sure that the cells corresponding to equivalent tiles are also 
translationally equivalent. This will usually require increasing the set of labels, by ``decorating'' each point with the translation equivalence class of a sufficiently large neighborhood. 
It is always possible to achieve this, see \cite[Section 4]{NPF}. Below we assume that this has already been done and the equations \eqref{subs1},\eqref{subs2} still apply, possibly with a larger $m$. 
Now let $\Lam = \Lam(\Tc)$ be from \eqref{subs101} and consider the tiling $\Tk_\Lam$, where the tiles ``inherit'' the labels from the corresponding control points. It can be  shown, as in \cite{NPF}, that, assuming that the labels carry the information about a sufficiently large neighborhood, the tiling $\Tk_\Lam$ is MLD with the original $\Tc$ and hence the translation dynamics are conjugate. 
If $\Tc$ is self-similar, then $\Tk_\Lam$, with $\Lam = \Lam(\Tc)$, is pseudo-self-similar { with the same expansion map}. We have 
$$
\Tk_\Lam = \bigcup_{k=1}^m \bigl(T_k + \Lc_k(\Tc)\bigr), 
$$
where
$$
\supp(T_k) = \supp\bigl(V_{\Lam-x}(c(\FrT_k)\bigr)\ \ \mbox{for}\ x\in \Lc_k(\Tc),
$$
and consistency (independence of $x$) is guaranteed by construction. For the pseudo-self-similar tiling $\Tk_\Lam$ we have a substitution, ``inherited'' from $\om$, which we denote by
the same letter by abuse of notation. In fact, similarly to \eqref{subs1} and \eqref{subs20},
\be \label{subs11}
\omega(T_j) = \bigcup_{k\le m} \bigl(T_k + \Dk_{jk}\bigr),\ \ \ \ \om(T_j + x) = \om(T_j) + \varphi(x),\ \ \ j\le m,\ x\in \R^d,
\ee
but the analog of \eqref{subs2} does not hold exactly -- only approximately. Then the pseudo-self-similar (PSS) tiling satisfies $\om(\Tk_\Lam) = \Tk_\Lam$.
The atlas of patches of $\Tk_\Lam$, or equivalently, its orbit closure in the natural topology, defines the PSS tiling space.
By construction, the PSS tiling space is MLD with the initial self-similar tiling space.

{
We summarize this discussion in the following proposition, but first we need some terminology.

\begin{defi} \label{def-Poly}
A tiling of $\R^d$ will be called {\em polytopal} if its tiles are convex polytopes meeting face-to-face, and this induces a structure of a CW-complex.
\end{defi}

\begin{defi}\label{def-RL}
A repetitive FLC PSS tiling $\Tk$, with the prototile set $\{T_j\}_{j\le m}$,
an expansion map $\varphi$, and a (combinatorial-geometric) substitution $\om$, will be called an {\em L-PSS tiling} (L is for ``lucky'') if it is polytopal and 
there exists a substitution Delone multi-set $\bigl(\Lc_k\bigr)_{k\le m}$ with expansion $\varphi$, satisfying \eqref{eq-SDel}, such that 
$\Tk = \bigcup_{k=1}^m \bigl(T_k + \Lc_k\bigr)=\om(\Tk)$ and $\om$ acts by \eqref{subs11}.
\end{defi}

We have proved the following

\begin{prop} \label{prop-RL}
If $\Tk'$ is a self-similar tiling with an expansion map $\varphi$, then there exists an L-PSS tiling $\Tk$ with the expansion map $\varphi$, such that
$\Tk$ and $\Tk'$ are MLD.
\end{prop}

}

\begin{remark} \label{rem-pss} {\em
(a) Although the procedure described above is general, for specific examples there is frequently a direct and relatively simple way to obtain a PSS tiling from a self-similar one without increasing the number of prototiles. In fact, often a planar PSS tiling is given directly, equipped with a ``substitution-with-amalgamation'', see \cite{CS06}, and a procedure of ``redrawing the boundary'' is
used to show that it is MLD to a genuine self-similar one \cite{FraSol}. This is the case for the class of Kenyon's tilings \cite{Kenyon} which we analyze in Section~\ref{sec:ex}.

(b) The step of passing from a self-similar tiling to a PSS tiling is not needed if the self-similar tiling is polytopal to begin with.

(c)  It is proved in \cite{FraSol} for tilings in the plane and in \cite{SolPSS} in the general case, that any PSS tiling with an expansion map $\varphi$ is MLD with a genuine
self-similar tiling, with the expansion map $\varphi^n$, for some $n\in \N$. Thus in the last proposition we could start with an arbitrary PSS tiling, at the cost of raising the associated
expansion map to a power.
}
\end{remark}


\section{Cohomology and deformations} \label{sec:coh}

\subsection{Pattern equivariant cohomology}
Let $\mathcal{T}$ be an aperiodic repetitive tiling of finite local complexity and recall the definition of a $\mathcal{T}$-equivariant function (\ref{PEfunction}). A $\mathcal{T}$-equivariant $k$-form is a $k$-form $\alpha$ such that its coefficients are $\mathcal{T}$-equivariant functions. We denote the set of $C^\infty$, $\mathcal{T}$-equivariant $k$-forms by $\Delta_\mathcal{T}^k$, which is a subspace of the set of smooth $k$-forms on $\mathbb{R}^d$. As such the restriction of the usual de Rham differential operator gives a differential operator on the complex $\{\Delta_\mathcal{T}^k\}_k$ of $\mathcal{T}$-equivariant forms.
\begin{defi}
  The cohomology of the complex of smooth $\mathcal{T}$ equivariant forms
  $$H^k(X;\mathbb{R}) := \frac{\ker d:\Delta_\mathcal{T}^k\rightarrow \Delta_\mathcal{T}^{k+1}}{\mathrm{im}\, d:\Delta_\mathcal{T}^{k-1}\rightarrow \Delta_\mathcal{T}^{k}}$$
  is called the \emph{$\mathcal{T}$-equivariant cohomology}.
\end{defi}
We denoted the cohomology as $H^k(X;\mathbb{R})$ since it is independent of which tiling $\mathcal{T}\in X$ we used to define it \cite{KellendonkPutnam}.

\subsection{\v Cech cohomology}
Let $\mathcal{T}$ be an aperiodic, repetitive tiling of finite local complexity, which we now assume has a CW structure. More specifically, we assume that all the tiles of $\mathcal{T}$ are $d$-cells of the CW complex $\mathcal{T}$ where all tiles meet face-to-face. { (In practice, we will work with polytopal tilings, see Definition~\ref{def-Poly}.)}
For any tile $t\in\mathcal{T}$ we define $\mathcal{T}(t) = \Tk^{(1)}$, called the {\em 1-corona of $t$}, to be the set of all tiles in $\mathcal{T}$ which intersect $t$. Continue recursively as follows: 
Given a $(k-1)$-corona of a tile $\mathcal{T}^{(k-1)}(t)$, the {\em $k$-corona} of the tile $t$ is the patch
$$\mathcal{T}^{(k)}(t) = \{t'\in\mathcal{T}: t'\cap\mathcal{T}^{(k-1)}(t)\neq \varnothing\}.$$
 A {\em $k$-collaring} of a tile $t$ is the tile with the same support as $t$, but the label being the translation-equivalence class of $\Tk^{(k)}(t)$.
This is a useful tool to keep track of bigger neighborhoods of tiles by increasing the set of labels.

\subsubsection{The Anderson-Putnam complex}

{
Following Anderson-Putnam \cite{AP}, we define a cell complex ${AP}_{0}(X)$ by gluing together the prototiles along their faces in all ways in which they can be adjacent in the tiling space.
This can be done equally well for $k$-collared tiles (this just increases the set of labels). The resulting space is denoted $AP_k(X)$. Here is a formal definition:}

\begin{defi}
  Let $X$ be a tiling space of repetitive, aperiodic tilings of $\mathbb{R}^d$ of finite local complexity. We define on $X\times\mathbb{R}^d$ an equivalence relation $\sim_1$, where $X$ carries the discrete topology and $\mathbb{R}^d$ the usual topology, as $(\mathcal{T}_1,v_1)\sim_1(\mathcal{T}_2,v_2)$ if and only if $\mathcal{T}_1(t_1)-v_1 = \mathcal{T}_2(t_2)-v_2$ for some tiles $t_1,t_2$ with $v_1\in t_1\in \mathcal{T}_1$ and $v_2\in t_2\in \mathcal{T}_2$. The space $X\times \mathbb{R}^d/\sim_1$ is the \emph{Anderson-Putnam complex} of $X$, denoted by $AP(X)$. 
  For $k\geq 0$ the $k^{th}$-collared Anderson-Putnam complex $AP_{k}(X)$ is similarly obtained by using $k$-collared tiles instead of $1$-collared tiles, where 0-collared tiles are just tiles.
  { Note that $AP(X) = AP_1(X)$ by definition.}
\end{defi}

Assuming a tile substitution rule as in {(\ref{subs11})}  is a cellular map, it defines a map $\gamma:AP(X)\rightarrow AP(X)$. The important observation of Anderson and Putnam is that, as defined,
$$X \cong \lim_{\longleftarrow}(AP(X),\gamma),$$
that is, the tiling space $X$ is homeomorphic to the set of all infinite sequences $\{x_i\}_{i\in\mathbb{N}}\in AP(X)^\infty$ with the property that $\gamma(x_i) = x_{i-1}$. What this result gives is the easy calculation of the \v Cech cohomology of $X$ using the induced maps $\gamma^*$ on the cohomology of $AP(X)$:
\begin{equation}
  \label{directLimit}
\check H^*(X;\mathbb{Z})\cong \lim_{\longrightarrow} \left(\check H^*(AP(X);\mathbb{Z}\right),\gamma^*).
\end{equation}

\subsubsection{Cohomology for PSS tilings}
\label{subsubsec:cohPSS}
Let $\mathcal{T}$ be a { polytopal}
 pseudo-self-similar tiling. The latter
 means that there exists an $R>0$ and expansive map $\varphi$ such that for all $x,y\in\mathbb{R}^d$,
$$\mathcal{O}^-_\mathcal{\varphi T}(B_R(x)) = \mathcal{O}^-_\mathcal{\varphi T}(B_R(x)) +(x-y) \Longrightarrow \mathcal{O}^-_\mathcal{T}(B_1(x)) = \mathcal{O}^-_\mathcal{ T}(B_1(x)) +(x-y).$$
Let $X$ be the tiling space of $\mathcal{T}$ and $X_\varphi$ the tiling space of $\varphi\mathcal{T}$. If $\varphi$ is pure dilation, then $AP(X_{\varphi})$ is a rescaled copy of $AP(X)$ by $\theta:= |\det \varphi|^{1/d}$, since linear expansive maps do not affect the process of collaring. If $\varphi$ is not pure dilation, then $AP(X_{\varphi})$ is a rescaled copy of $AP(X)$ by $\theta$, but where all the cells have also been rotated.

\begin{prop}
  \label{PSSinverse}
  Let $X$ be the tiling space of a {polytopal} pseudo-self-similar tiling $\mathcal{T}$. 
  Then there exists a $\kappa\in\mathbb{N}$ and a map $\gamma:AP_\kappa(X)\rightarrow AP_\kappa(X)$ such that
  \begin{equation}
    \label{AP}
    X\cong \lim_{\longleftarrow} \left(AP_\kappa(X),\gamma \right).
  \end{equation}
\end{prop}
\begin{proof}
  Since $\mathcal{T}$ is PSS with expanding map $\varphi$, there exists an $R>0$ such that for any $x\in\mathbb{R}^d$ the $R$-neighborhood of $x$ in $\varphi\mathcal{T}$ determines the tile(s) to which $x$ belongs in $\mathcal{T}$. Let $k'\in\mathbb{N}$ be the smallest integer $k$ such that the $R$-neighborhood of any tile $t\in \varphi\mathcal{T}$ is contained in $(\varphi\mathcal{T})^{(k)}(t)$. We will first show that any $\ell\geq k'$ allows us to define a map $\gamma_\ell:AP_{\ell}(X)\rightarrow AP_{\ell}(X)$.
  
  Let $X_\varphi$ be the tiling space of $\varphi\mathcal{T}$. Since the pseudo self-similarity $\varphi$ maps $n$-cells of $\mathcal{T}$ to $n$-cells of $\varphi\mathcal{T}$ and respects collaring, it determines a bijection $r_\ell:AP_\ell(X)\rightarrow AP_\ell(X_\varphi)$ for every $\ell\in\mathbb{N}$. We now define a map $s_\ell:AP_{\ell}(X_\varphi)\rightarrow AP_{\ell}(X)$ for any $\ell\geq k'$ as follows. Let $[x]_\varphi\in AP_{\ell}(X_\varphi)$ and let $x \in t\in \varphi\mathcal{T}\in X_\varphi$ be a representative on the tiling. Then the $R$ neighborhood of $x$ in $\varphi \mathcal{T}$ determines a point in $\mathcal{T}$ whereon $x$ lies, and thus a point $[x]\in AP_{\ell}(X)$. Since $[x]$ was completely determined by the $R$-neighborhood of $x$ and this neighborhood is completely contained in $(\varphi\mathcal{T})^{(\ell)}$, then the map $s_\ell([x]_\varphi) = [x]\in AP_{\ell}(X)$ is well-defined (that is, independent of representatives of classes). Let $\gamma_\ell:= s_\ell\circ r_{\ell}:AP_{\ell}(X)\rightarrow AP_{\ell}(X)$.
  
  Now the inverse limit of $(AP_{k'}(X),\gamma_{k'})$ is well defined, but it may or may not be homeomorphic to $X$. By the argument of Anderson and Putnam, collaring once more guarantees that the map forces the border. Therefore, picking $\kappa = k'+1$ ensures that and we obtain (\ref{AP}).
\end{proof}

\begin{corollary}
  Let $X$ be the tiling space of a { polytopal} pseudo-self-similar tiling $\mathcal{T}$ and $\gamma: AP_\kappa(X)\rightarrow AP_\kappa(X)$ the map from Proposition \ref{PSSinverse}. Then
  \begin{equation}
    \label{APcoh}
    \check H^*(X;\mathbb{Z})\cong \lim_{\longrightarrow} \left(\check H^*(AP_\kappa(X);\mathbb{Z}),\gamma^* \right).
  \end{equation}
\end{corollary}
Finally, we note that for polytopal tilings of finite local complexity in which cells meet face to face, we have that the $\mathcal{T}$-equivariant and \v Cech cohomologies are isomorphic
\cite{KellendonkPutnam}, that is,
\begin{equation}
  \label{CohIso}
  H^*(X;\mathbb{R})\cong \check H^*(X;\mathbb{R}).
\end{equation}

\subsection{Deformations}
\label{subsec:deform}
In this section we go over deformations of tiling spaces, following \cite{CS06,JS}. Let $\mathcal{T}$ be a repetitive, aperiodic { polytopal} tiling of finite local complexity 
 and let $X$ be the corresponding tiling space.
  Observe that  for L-PSS tiling spaces $H^1(X;\mathbb{R})$ is finite dimensional by \eqref{APcoh}.
 Each class $[\alpha]\in H^1(X;\mathbb{R})$ is represented by a $\mathcal{T}$-equivariant smooth 1-form $\alpha:\mathbb{R}^d\rightarrow T^*\mathbb{R}^d$ (up to a $\mathcal{T}$-equivariant exact form). Now consider the space $H^1(X;\mathbb{R}^d) = H^1(X;\mathbb{R})\otimes \mathbb{R}^d$. Each class $[\fra]\in H^1(X;\mathbb{R}^d)$ is represented by a $\mathcal{T}$-equivariant smooth 1-form $\fra:\mathbb{R}^d\rightarrow T^*\mathbb{R}^d\otimes \mathbb{R}^d \cong M_{d\times d}$ (up to a $\mathcal{T}$-equivariant exact form). That is, the representative $\fra$ is a $\mathcal{T}$-equivariant choice of linear transformation of $\mathbb{R}^d$.

A representative $\fra$ of a class in $H^1(X;\mathbb{R}^d)$ is called \emph{shape deformation} as it defines a deformation of $\mathcal{T}$ as follows. (We essentially quote the beginning of \cite[Section 8]{JS} here.)
 Suppose that $\fra$ is $\Tk$-equivariant with some radius $R$. Let $H_\fra:\mathbb{R}^d\rightarrow \mathbb{R}^d$ be defined as $H_\fra(x) = \int_0^x\fra$. We will deform the {\em vertices} of tiles, {\em assuming that one of the vertices is at the origin}. Each of these vertices is decorated with the equivalence class of the pattern of  radius $R_0>R$
around it, where $R_0$ is greater than $R$ plus the greatest distance between the  adjacent vertices (i.e., connected by an edge).
If $x$ is a vertex, we let $H_\fra(x)$ be the 
deformed vertex (with the same label). 
If $v_1 v_2$ was an edge, then the displacement vector $v_2 - v_1$ has been changed to $\int_{v_1}^{v_2} \fra$. The label of $v_1$ determines the pattern of $\Tk$ out to distance $R$ along the
entire edge  and hence determines $\int_{v_1}^{v_2} \fra$. Thus the local patterns of $\Tk^\fra$ are determined from local patterns of $\Tk$, and $\Tk^\fra$ has FLC.

 Since we deform the  tiles by deforming their edges, for $d>2$ we need to make sure that higher-dimensional faces are well-defined. For instance, for 2-dimensional faces we would need the side edges to remain co-planar. In order to avoid imposing such a condition, we can simply triangulate all the tiles to begin with. It is easy to see that there exists a triangulation which is equivariant and preserves the face-to-face property. When all the faces are simplices, they stay being simplices after perturbing the edges. The new simplex sub-tiles should carry a label which contains both the label of the tile it came from, and its location in the triangulation; then the triangulated tiling is MLD to the original one, and we can work with it from the start
 without loss of generality.

What is not clear at this point is whether the deformation is in fact a homeomorphism or not. As proved by \cite{JS}, what determines the answer to this question is invertibility of the Ruelle-Sullivan cycle applied to $[\fra]$.
\begin{defi}
  Let $\mathcal{T}$ be a { polytopal}, repetitive, aperiodic tiling of finite local complexity,  with uniform patch frequency. 
  The \emph{Ruelle-Sullivan map} is the map $C:H^1(X;\mathbb{R}^d)\rightarrow M_{d\times d}$ defined, for a class $[\fra] \in H^1(X;\mathbb{R}^d)$, as
  $$C([\fra]) = \lim_{R\rightarrow \infty} \frac{1}{\mathrm{Vol}(B_R)}\int_{B_R}\fra(t)\, dt,$$
 which is independent of the representative $\fra$.
\end{defi}
We can now recall one of the main results of \cite{JS}:  if $C([\fra])$ is invertible, then one can choose a representative $\fra$ so that $H_\fra$ is a homeomorphism of $\R^d$. 
Then the shape deformation $\fra$ induces an orbit equivalence between the tiling spaces $X$ and $X^\fra$. In particular, the spaces $X$ and $X^\fra$ are homeomorphic. As such, the set
$$\mathcal{M}(X) := \left\{[\fra]\in H^1(X;\mathbb{R}^d):\det C([\fra])\neq 0\right\}$$
parametrizes deformations of the tiling space $X$, up to deformations given by coboundaries. Deformations given by representatives of classes in $\mathcal{M}(X)$ are all orbit equivalences.
Moreover, $\fra$ can be chosen so that the  homeomorphism $H_\fra$ will be arbitrarily close to the identity, see the proof of \cite[Theorem 8.1]{JS}.
This, however, will come at the cost of collaring to a possibly very large radius.

It is important for us that the deformation preserves the combinatorial structure of the tiling, in the sense neighbor graph and faces of all dimensions. This can be achieved either by collaring, or by
working only with super-tiles of sufficiently large size and ignoring the ``small-scale'' combinatorics.

\begin{defi}
A shape deformation of an FLC polytopal tiling space $X$ will be called {\em admissible} if it preserves the local combinatorial structure of the tilings {and it defines a homeomorphism via the map $H_{\fra}$ introduced above}. In practical terms this means that we will consider
deformations that are small compared to the size of the prototiles.
\end{defi}

Now consider the tiling space $X_\omega$ of an
 { L-PSS} tiling, see Definition~\ref{def-RL}. Picking $[\fra]\in\mathcal{M}(X_\omega)$ and applying an { admissible} $\fra$ to $\Sc\in X_\omega$ {having a vertex at the origin},
 we obtain a tiling $\Sc^\fra$ and the $\fra$-deformed substitution tiling space $\X^\fra_{\omega}$ (the translation orbit closure), whose prototiles will be denoted by 
$T_j^\fra$.  To be precise, we need to specify the {\em location} of $T_j^\fra$ in $\R^d$. By construction, the L-PSS tiling $\Tk$ has distinguished prototiles $T_k\in \Tk$ and is also 
a fixed point of the substitution $\om(\Tk) = \Tk$. Choose any vertex $v$ in $\Tk$ and consider the shifted tiling $\Tk-v$, with a vertex at the origin. Then the deformation
$(\Tk-v)^\fra$ is well-defined.
We then let the {\em deformed prototiles} to be
\be \label{def-prot}
T_j^\fra:= (T_j - v)^\fra,\ \ j\le m.
\ee
Further, $\om^n(T_j) - v \subset \Tk-v$ for all $n\in \N$, so we can define {\em deformed higher order super-prototiles} by
\be \label{def-prot2}
T_j^{\fra,n} = (\om^n(T_j) - v)^\fra.
\ee
(Strictly speaking, these are patches rather than individual tiles. The exact location is not that important, but we want their supports to be subsets of $\R^d$, e.g., in order to be able to perform
integration over them.) {\ The assumption that $\fra$ is admissible implies that the tilings in the deformed tiling space $X_\om^\fra$ will have a hierarchical structure combinatorially 
equivalent to those of
$X_\om$.}

\medskip

\subsection{Recurrences, return vectors and recurrence vectors} \label{sec-recur}
{ Let $\Tk$ be a repetitive FLC polytopal tiling and $X$ the corresponding tiling space.}
 Following \cite{CS06}, we say that a pair of points $(z_1,z_2)$ in a tiling $\Tk$ is a {\em recurrence} of size $r$, if 
\be \label{eq-recur1}
\Ok^-_\Tk(B_r(z_2)) = \Ok^-_\Tk(B_r(z_1)) + (z_2 - z_1),
\ee
and $r$ is maximal possible.  { The vector $z_2-z_1$ is called a {\em return vector} of size $r$.}
We will always assume that $r\ge 2D_{\max}$, where $D_{\max}$ is the diameter of the largest prototile, then
\eqref{eq-recur1} implies that $\Ok^+_\Tk(\{z_2\}) = \Ok^+_\Tk(\{z_1\}) + (z_2 - z_1)$.
We can assume, without loss of generality, that $z_1$ and $z_2$ are vertices of the tiling $\Tk$. Each path along edges from $z_1$ to $z_2$ projects to a
closed loop in $AP(X)$, and hence to a closed chain in $C_1(AP(X))$. Different paths from $z_1$ to $z_2$ project to homologous chains. The class in $H_1(AP(X),\Z)$ of a recurrence is called
a {\em recurrence class}. Recurrences of size greater than $(k+1)D_{\max}$ project to closed chains in $AP_k(X)$ and define classes
in $H_1(AP_k(X),\Z)$.

 {

\begin{defi} \label{def-elem} Let $\Tk$ be a polytopal tiling. A pair of points $(z_1,z_2)$ in $\Tk$ will be called an {\em elementary recurrence} if there exists a tile $T\in \Tk$ such that
$z_1\in \supp(T)$ and $T+ (z_2-z_1) \in \Tk$.
Similarly to the above, we can associate to $(z_1,z_2)$ a closed loop in ${AP}_{0}(X)$ which defines a class in $H_1({AP}_0(X),\Z)$, called an {\em elementary recurrence class}.
\end{defi}

\begin{defi} \label{def-elem2}
  Let $[\fra]\in\mathcal{M}(X) \subset H^1(X;\R^d)$. We will call it {\em elementary} if it is a pull-back of a class in $H^1({AP}_0(X),\R^d)$ under the ``forgetful'' map. In particular, it is a
class of a shape deformation which acts on elementary recurrences. {A representative $\fra$ of an elementary class $[\fra]$ will also be called elementary.}
\end{defi}

In the rest of the paper we will restrict ourselves to elementary shape deformations. This is not an essential restriction, since we can always ``collar'' the tiles of $\Tk$ to the level $\kappa$,
given in Proposition \ref{PSSinverse}, and then an elementary deformation defines a class in $H^1(AP_\kappa(X);\R^d)$, hence in $H^1(X;\R^d)$, by Proposition \ref{PSSinverse}.

}

We continue with the construction, following \cite{CS06}. Let $\X$ be an L-PSS tiling space.
The set of integer linear combinations of { elementary recurrence classes} is a subgroup $\Gam < H_1({AP}_0(X),\Z)$, 
which is a finitely generated free $\Z$-module. Let $\{a_1,\ldots,a_s\}$ be a basis (set of free generators) for $\Gam$. For each elementary recurrence $(z_1,z_2)$, which defines a class
$[(z_1,z_2)]$, the corresponding
{\em recurrence vector} in $\Z^s$ is the decomposition of the class in this basis; it will be denoted $\alpha(z_1,z_2)$. The transformation $\alpha:\Gam\to \Z^s$ is sometimes called the
{\em address map}.
For a shape parameter $\fra$, define
\be \label{def-Lf}
{\bf L}_\fra = (\fra(a_1),\ldots, \fra(a_s)),
\ee
which can be thought of as a $\R^d$-valued row vector that gives the displacements of the deformed vectors in the basis. Then the deformation of the displacement vector $z_2-z_1$, corresponding to the recurrence vector $\bv=\alpha(z_1,z_2)$
 is equal to ${\bf L}_\fra \bv$.
  
 Given an elementary recurrence $(z_1,z_2)$ in $\Tk$, we have that $(\varphi(z_1),\varphi(z_2))$ is an elementary recurrence as well, because $\Tk$ is a PSS tiling with expansion 
 $\varphi$, see \eqref{subs11}. (In fact, applying the substitution increases the size of the recurrence, but is still an elementary too.)
  It follows that the expansion map $\varphi$ induces an endomorphism of the $\Z$-module generated by elementary recurrence
 vectors. Thus there exists an integer $s\times s$ matrix $M$ satisfying
 \be \label{lin-map}
\alpha(\varphi(z_1),\varphi(z_2)) = M \alpha(z_1,z_2),\ \ z_1, z_2 \in \Lam(\Tk).
\ee 

Let $\fra$ be an elementary admissible deformation.
Then, combining \eqref{subs11} and \eqref{def-prot2} (for $n=1$) yields
\be \label{subs-f1}
T_j^{\fra,1} = \bigcup_{j\le m} \bigl(T_k^\fra + \fra(\Dk_{jk})\bigr),
\ee
where the right-hand side is a patch of a tiling in $X_\om^\fra$.
Here we view the elements of $\Dk_{jk}$ as elementary recurrences, since they represent the prototile $T_k\in \Tk$ and its translated copy in $\om(\Tk) \subset \Tk$.
Using the address map, as above, we get an associated set of recurrence vectors, denoted $\alpha(\Dk_{jk})$. By definition,
$$
\fra(\Dk_{jk}) = {\bf L}_\fra\alpha(\Dk_{jk}).
$$
We can iterate the procedure and obtain the decomposition of higher order deformed super-tiles as well.
For instance,
\begin{eqnarray} 
T_j^{\fra,2} & = & \bigcup_{s\le m, k\le m} \bigl(T_s^\fra + \fra(\varphi \Dk_{jk}) + \fra(\Dk_{ks})\bigr), \nonumber \\[1.2ex]
& = &  \bigcup_{s\le m, k\le m} \bigl(T_s^\fra + {\bf L}_\fra M \alpha (\Dk_{jk}) + {\bf L}_\fra \alpha (\Dk_{ks})\bigr),\ \ 1 \le j \le m. \label{subs-f2}
\end{eqnarray}

\subsection{Geometric properties of deformed tilings and consequences}

The following geometric lemma will be useful. We will write $\asymp$ to indicate that the equality holds up to a uniformly bounded from $0$ and $\infty$ multiplicative constant.

\begin{lemma} \label{lem-qi}
Let $\omega$ be an  { L-PSS} tile substitution with expansion $\varphi$ and $X_\om$ the corresponding tiling space.
Fix an elementary admissible shape deformation $\fra$, and consider the corresponding deformed tiling space $X_\om^\fra$. For the $\fra$-deformed supertiles 
$
T_j^{\fra,n},\ j\le m,
$
let $R_n^{\fra}$ be the radius of the smallest ball containing (a translate of) every $T_j^{\fra,n}$, and let $r_n^{\fra}$ be the radius of largest ball contained in a (translate of) any $T_j^{\fra,n}$.
Then there exists $C_\fra>1$ depending only on $\fra$ such that
\be \label{roundish1}
r_n^\fra\ge C_\fra^{-1} \theta^n 
\ee
and
\be \label{roundish2}
 R^\fra_n \le C_\fra \theta^n, 
\ee
where $\theta=\|\varphi\|$.
\end{lemma}

\begin{proof}
We use the fact that, combinatorially, the hierarchical structure of the tilings in $X_\om^\fra$ is the same as in $X_\om$.
The L-PSS tile substitution was obtained from a self-similar, ``geometric'' tile substitution, so that $\Tk$ is MLD with a self-similar tiling $\Tc$ having the expansion $\varphi$.
For the tiles and super-tiles of $\Tc$ the properties \eqref{roundish1}
and \eqref{roundish2} are immediate by self-similarity. It follows that the same properties hold the for the PSS tiling, since it was built using a derived Voronoi tesselation construction, using
$\Lam=\Lam(\Tc)$.

Let us call a sequence of tiles $S_1,\ldots,S_k$ (of arbitrary type) a {\em chain} if $S_j\cap S_{j+1}\ne \es$ for $j=1,\ldots,k-1$ (recall that our tiles are compact sets which share faces).

\medskip

{\sc Claim 1.} {\em Let $T_j$ be a prototile of the L-PSS tiling $\Tk$, with $c(T_j)$ (control point) in its interior, 
and consider the $\Tk$-patch $\omega^n (T_j)$, as well as $A_j^n := \supp (\omega^n (T_j))$.
Let $S$ be any $\Tk$-tile in the patch $\omega^n (T_j)$ containing the point $\varphi^n(c(T_j))$. Then any chain of $\Tk$-tiles in 
$\omega^n (T_j)$ connecting $S$ to $\partial A_j^n$ contains at least $c_1 \theta^n$ tiles, where $c_1>0$ depends only on $\Tk$.}

\medskip

The claim is immediate since $A^n_j$, being a bounded displacement of $\varphi^n A_j$, with $A_j = \supp(T_j)$, contains a ball of radius $\asymp \theta^n$ centered at $\varphi^n(c(T_j))$, by the \eqref{roundish1} property for $\Tk$.

\medskip

Now consider the $\fra$-deformed super-tile  $T_j^{\fra,n} = (\omega^n (T_j)-v)^\fra$, see \eqref{def-prot2}, with the deformed tile $(S-v)^\fra$ inside. Denote $A_j^{\fra,n} := 
\supp(T_j^{\fra,n})$.

\medskip

{\sc Claim 2.} {\em Let $d((S-v)^\fra,\partial A_j^{\fra,n})$ be the usual Euclidean distance between the compact boundary of the deformed super-tile and the compact $(S-v)^\fra$ in its interior.
Then there exists a chain of $\fra$-deformed tiles in the patch $T_j^{\fra,n}$ connecting $(S-v)^\fra$ to $\partial A_j^{\fra,n}$ of cardinality $N_\fra(n)$, such that
\be \label{f-chain}
N_\fra(n)\cdot V_{\min}^\fra \le c_{d-1}\cdot \bigl[d((S-v)^\fra,\partial A_j^{\fra,n}) +2\doh^\fra_{\max}\bigr] \cdot (\doh^\fra_{\max})^{d-1},
\ee
where $V^\fra_{\min}$ is the minimal volume of an $\fra$-deformed tile, $\doh^\fra_{\max}$ is the maximal diameter of an $\fra$-deformed tile, and $c_{d-1}$ is the volume of a unit ball in $\R^{d-1}$.}

\medskip

Combining Claim 2 with Claim 1 yields  \eqref{roundish1}. Indeed, for every $\fra$-chain in $T_j^{\fra,n}$ there is a combinatorially equivalent chain in the pseudo-self-similar
$\omega^n(T_j)$, and this yields a lower bound $\asymp \theta^n$ for $d((S-v)^\fra,\partial A_j^{\fra,n})$ for large enough $n$. (For small $n$ both \eqref{roundish1} and \eqref{roundish2} are trivial.)

\smallskip

For the proof of Claim 2, consider the shortest straight line segment $J$ from $(S-v)^\fra$ to $\partial A_j^{\fra,n}$ and consider the union of deformed tiles in $T_j^{\fra,n}$ intersecting $J$.
One can certainly form a chain of $\fra$-deformed tiles connecting $(S-v)^\fra$ to $\partial A_j^{\fra,n}$ out of them, and their total volume is at least the expression in the left-hand side of 
\eqref{f-chain}. On the other hand, these tiles are all contained in the $\doh^\fra$-neighborhood of $J$, whose volume is less that the right-hand side of \eqref{f-chain}.

\medskip

The proof of  \eqref{roundish2} is similar. Indeed, the maximal distance from a point in $S$ to a point on the boundary $\partial A^n_j$ is at most $C_2\theta^n$ by 
self-similarity, hence for any point on $\partial A^n_j$ there exists a chain of $\Tk$-tiles connecting that point to $S$, of cardinality at most $\wt C_2\theta^n$. For any such chain there is a corresponding deformed chain of the same cardinality, showing that the diameter of $A_j^{\fra,n}$ is bounded above by $\const\cdot \theta^n$, with the constant depending only on the deformation. This completes the proof of the lemma.
\end{proof}

For each prototile $T_j^\fra$ of the deformed tiling space $X_\om^\fra$ we also fix a control point $c(T_j^\fra)$ in the interior and then in all tiles at the same relative location. For a tiling $\Sc^\fra\in X_\om^\fra$ let
$$
\Lam(\Sc^\fra) := \{c(T^\fra):\ T^\fra\in \Sc^\fra\}.
$$

\begin{lemma} \label{lem2-qi}  Let $\Tk\in X_\om$ be an L-PSS tiling,  $\fra$ and admissible elementary deformation, and $\Sc^\fra = (\Tk-v)^\fra\in X_\om^\fra$, where $v$ is a vertex of 
$\Tk$. 
 Then $\Lam(\Tk)$ and $\Lam(\Sc^\fra)$ are quasi-isometric; in fact, there exists $C_{\omega,\fra}$ (independent of $\Tk$) such that for any
 $z_1, z_2 \in \Lam(\Tk)$ and the corresponding $z_1^\fra, z_2^\fra \in \Lam(\Sc^\fra)$, holds
 \be \label{eq1-qi}
C_{\omega,\fra}^{-1} |z_1 - z_2| \le |z_1^\fra - z_2^\fra| \le C_{\omega,\fra} |z_1 - z_2|.
 \ee
\end{lemma}

\begin{proof}
Consider the graph ${\rm Graph}(\Tk)$ associated with a tiling $\Tk$, in which the vertices are the tiles and graph edges connect neighboring tiles (tiles whose compact supports intersect). By construction, this graph is stable under the deformations that we consider. It follows from the proof of
Lemma~\ref{lem-qi} that for the deformed tiling $\Sc^\fra$, the set $\Lam(\Sc^\fra)$ is quasi-isometric to  the set of vertices of the graph ${\rm Graph}(\Sc^\fra) \cong {\rm Graph}(\Tk)$, endowed with the graph metric. This implies the desired claim.
\end{proof}

\begin{defi}
A tiling $\Tk$ is called {\em linearly repetitive} if there exists $C>0$ such that for every patch $\Pk\subset \Tk$ a translationally equivalent $\Tk$-patch may be found in every ball of radius
$C\diam(\Pk)$.
\end{defi}

\begin{corollary} Let $\X_\om$ be an FLC primitive  aperiodic L-PSS tiling space. 
Fix an admissible elementary shape deformation $\fra$ and consider the corresponding deformed tiling space
$X_\om^\fra$. 
Then

{\rm (i)} all the tilings in $ \X^\fra_{\omega}$ are linearly repetitive;

{\rm (ii)} the deformed tiling dynamical system $( \X^\fra_{\omega}, \R^d)$ is uniquely ergodic.
\end{corollary}

\begin{proof}
Linear repetitivity of primitive self-similar tiling spaces is well-known, see, e.g. \cite{SolUCP}, and the last lemma implies that this is the case for $\X^\fra_\om$. Linear repetitivity
implies the existence of uniform patch frequencies \cite{LaPl} and hence unique ergodicity.
\end{proof}

The unique invariant probability measure for $( \X^\fra_{\omega}, \R^d)$ will be denoted by $\mu_\fra$.

\begin{remark}{\em
The deformed tiling space $X_\om^\fra$ falls into the very general framework of {\em fusion}, developed by Sadun and Frank \cite{FraSa1}. In fact, our situation satisfies the conditions of 
transition-regular, primitive, and recognizable fusion rules, for which unique ergodicity was proved in \cite[Corollary 3.10]{FraSa1}. Moreover, we have a constant number of prototiles (equal to $m$) at each level, and a constant transition matrix, hence the tiling dynamical system $( \X^\fra_{\omega}, \R^d)$, with the unique invariant probability measure $\mu_\fra$, is not strongly
mixing by \cite[4.13]{FraSa1}.
}
\end{remark}

We will need a formula for the unique invariant measure $\mu_\fra$ on  $\X^\fra_\omega$. Recall the notion of {\em transversal}, defined by
$$
\Upsilon(\Pk^\fra) = \{\Tk^\fra\in \X^\fra_\omega:\ \Pk^\fra \subset \Tk^\fra\}
$$
for a deformed tiling space $X^\fra_\om$ and a deformed patch $\Pk^\fra$.

\begin{lemma} \label{lem-measure}
Let $\X_\om$ be an FLC primitive  aperiodic L-PSS tiling space, with expansion $\varphi = \theta \Ok$. Fix an elementary admissible shape deformation $\fra$ and consider the corresponding deformed tiling space $X_\om^\fra$. 
Then there exists $c_{\om,\fra}>0$, such that for any Borel set $U$ and all $n\in \N$,
\be \label{eq2-meas}
\diam(U) \le c_{\om,\fra} \theta^n \implies \mu_\fra\bigl(\Upsilon(T_j^{\fra,n})+U\bigr) \ge \const\cdot \Lk^d(U)\cdot \theta^{-nd},
\ee
with the  constants depending only on $\om$ and on $\fra$.
\end{lemma}

\begin{proof}
It is proved in \cite[Lemma 2.4]{SolUCP} that for an aperiodic primitive self-similar tiling $\Tc$, there exists $0< c < 1$ such that if $\Pk\subset \Tc$ is a patch containing a ball of radius $R>0$ in its support, then $\Pk + x\not\subset \Tc$ for all $x$ with $\|x\| \le c R$. Since the Delone set of the deformed tiling $\Sc^\fra=(\Tk-v)^\fra$ 
is quasi-isometric to $\Lam(\Tk)=\Lam(\Tc)$ by Lemma~\ref{lem2-qi}, this
property persists, with an appropriate constant, for $\Sc^\fra$. By \eqref{roundish1}, the patch $T_j^{\fra,n}$ contains a ball of radius $\asymp \theta^n$ in its support. 
We are going to use Proposition~\ref{prop-meas1}. Represent the set $U$ as a union of disjoint Borel sets $U_k$, each of diameter less than $\eta(\Sc^\fra)=2r_\fra$, the maximal diameter of a ball contained in the support of every deformed tile.
It follows that 
if $\diam(U) \le c_{\om,\fra} \theta^n$, with a sufficiently small constant, independent of $n$, then the sets $\Upsilon(T_j^{\fra,n})+U_k$ are mutually disjoint, and hence
by \eqref{eq-meas1},
\be 
\mu_\fra\Bigl(\Upsilon(T_j^{\fra,n})+U\Bigr) = 
\sum_k \mu_\fra\Bigl(\Upsilon(T_j^{\fra,n})+U_k \Bigr) = 
\sum_k \Lk^d(U_k) \cdot \freq(T_j^{\fra,n}) = 
\Lk^d(U) \cdot \freq(T_j^{\fra,n}).
\ee
By primitivity of the substitution and the Perron-Frobenius Theorem, the frequency of $\om^n(T_j)$ in the self-similar tiling space is at least $\const\cdot \theta^{-nd}$ (in fact, an upper bound
holds as well, using \cite[Lemma 2.4]{SolUCP} again, but we do not need it), just by counting the number of $n$-level super-tiles inside $(n+k)$-level super-tiles. By quasi-isometry, the same
asymptotic bound holds for the deformed tiling space, and \eqref{eq2-meas} follows.
\end{proof}

\smallskip


\section{The spectral cocycle; statement of result on local dimension} \label{sec:coc}

Let $X_\om$ be an FLC primitive aperiodic L-PSS tiling space. The {\em Fourier matrix} is an $m\times m$ complex matrix-function on $\R^d$ whose $(j,k)$-entry is 
\be \label{Four-mat}
[\Bc(\lamb)]_{(j,k)}:=\Bigl[\sum_{\bx\in \Dk_{jk}} \exp\bigl(-2\pi i \langle \lamb, \bx\rangle\bigr)\Bigr],\ \ \lamb\in \R^d,\ \ 1 \le j,k \le m,
\ee
see \cite[(29)]{BGM}.
First, we need to define a ``deformed'' Fourier matrix. Recall that
the set of integer linear combinations of elementary recurrence classes  is a subgroup $\Gam < H_1(AP_0(X_\om),\Z)$, for which we fixed a set of free generators
 $\{a_1,\ldots,a_s\}$. The decomposition with respect to this basis is given by $\alpha:\Gam\to \Z^s$, the address map.
Recall \eqref{lin-map} saying that the expansion induces a linear mapping on $\Z^s$ given by a matrix $M$, so that
$$
\alpha\varphi = M\alpha.
$$
For a shape parameter $\fra$ (a representative of a class in $H^1(AP_0(X_\om);\R^d)$), we considered the row vector
$
{\bf L}_\fra = (\fra(a_1),\ldots, \fra(a_s)),
$
see \eqref{def-Lf}.
{We also discussed that elements of $\Dk_{jk}$ can be identified with elementary recurrences, so $\alpha(\bx)\in \Z^s$ is well-defined for $\bx\in \Dk_{jk}$.}
Now we define the deformed Fourier matrix by
\be \label{defour-mat}
\Bigl[\sum_{\bx\in \Dk_{jk}} \exp\bigl(-2\pi i \langle \lamb,{\bf L}_\fra\alpha(\bx) \rangle \bigr)\Bigr]_{j,k\le m}
=\Bigl[\sum_{\bx\in \Dk_{jk}} \exp\bigl(-2\pi i \langle {\bf L}^{\sf T}_\fra\lamb, \alpha(\bx) \rangle_{_{\R^s}} \bigr)\Bigr]_{j,k\le m},\ \ \lamb\in \R^d.
\ee
The {\em spectral cocycle} will be over 
 the toral endomorphism
$$
\bz\mapsto M^{\sf T} \bz \ \mbox{mod} \ \Z^s,
$$
where the superscript ${\sf T}$ indicates the transpose.
It will be an endomorphism of the $s$-torus $\T^s$ if $\det(M)\ne 0$; otherwise, we can restrict it to a lower-dimensional invariant sub-torus.
Let
\be \label{spec-mat}
\Mc(\bz) := \Bigl[\sum_{\bx\in \Dk_{jk}} \exp\bigl(-2\pi i \langle \bz,  \alpha(\bx)\rangle_{_{\R^s}}\bigr)\Bigr]_{j,k\le m},\ \ \bz\in \R^s;
\ee
this is closely related to the deformed Fourier matrix, which was defined as $\Mc({\bf L}_\fra^{\sf T}\lamb)$, for $\lamb \in \R^d$, in \eqref{defour-mat}. 
Because of periodicity, $\Mc$ is well-defined on the torus $\T^s=\R^s/\Z^s$.

\begin{defi} \label{def-cocy}
The matrix product
$$
\Mc(\bz,n):= \Mc\bigl({(M^{\sf T})}^{n-1}\bz\bigr)\cdot\ldots\cdot \Mc(\bz),\ \ \bz\in \T^s,\ n\in \N,
$$
will be called the spectral cocycle associated with the L-PSS tiling space $X_\om$.

\end{defi}
We will see below that for
$\bz= {\bf L}_\fra^{\sf T}\lamb$, the growth behavior of $\Mc(\bz,n)$ in some sense ``controls'' the local behavior of spectral measures at $\lamb$.

As a consistency check, consider the case of unperturbed self-similar tiling, when ${\bf L}_\fra = [a_1,\ldots,a_s]$. Then $\bz =  [a_1,\ldots,a_s]^{\sf T}\lamb$ and
$$
\langle \bz,\alpha(\bx) \rangle = \langle \lamb,  [a_1,\ldots,a_s]\alpha(\bx)\rangle = \langle \lamb, \bx\rangle,
$$
by the definition of the address map $\alpha$. So we obtain $\Mc(\bz) = \Bc(\lamb)$ (the Fourier matrix) from \eqref{spec-mat}, \eqref{Four-mat}. Further,
\be \label{Four-coc}
 M^{\sf T}\bz = M^{\sf T} [a_1,\ldots,a_s]^{\sf T} \lamb =  [a_1,\ldots,a_s]^{\sf T}\varphi^{\sf T} \lamb, 
\ee
by the definition of $M$,
hence
$$
\Mc(\bz,n) = \prod_{j=0}^{n-1} \Bc\bigl((\varphi^{\sf T})^j \lamb\bigr),
$$
which agrees with the Fourier matrix cocycle of Baake et al., see \cite{BGM,Baake3}.

\smallskip

Define the {\em pointwise upper Lyapunov exponent} of the cocycle $\Mc(\bz,n)$ at the point $\bz\in \R^s$ by
\be \label{Lyap1}
\chi^+(\bz):= \limsup_{n\to \infty} \frac{1}{n} \log\|\Mc(\bz,n)\|;
\ee
we omit the superscript if the limit exists. In the ``unperturbed,'' self-similar case, it becomes
$$
\chi^+(\lamb) :=  \limsup_{n\to \infty} \frac{1}{n} \log\Bigl\|\prod_{j=0}^{n-1} \Bc\bigl((\varphi^{\sf T})^j \lamb\bigr)\Bigr\|.
$$
We will also need a more refined version of the Lyapunov exponent: for $\vec\zeta\in \C^m$, let
\be \label{refine1}
\chi^+(\bz,\vec\zeta):= \limsup_{n\to \infty} \frac{1}{n} \log\|\Mc(\bz,n)\vec\zeta\|;
\ee
Obviously, $\chi^+(\bz,\vec\zeta) \le \chi^+(\bz)$ for all $\vec\zeta$.

Observe that $\Mc({\bf 0})=\Sf_\om^{\sf T}$, the transpose substitution matrix of $\omega$, see \eqref{sub-matr}.
 Its PF eigenvalue is  $\vartheta_1 = |\det\varphi|=\theta^d$, where $\theta$ is the expansion constant, by self-similarity.
Thus $\chi({\bf 0}) = d\log \theta$. Since every entry of $\Mc(\bz)$ in absolute value
is less than or equal to the corresponding entry of $\Mc({\bf 0})$, a non-negative primitive matrix, we obtain that
$$
\chi^+(\bz) \le d\log\theta\ \ \mbox{for all} \ \bz\in \R^s.
$$

The {\em lower local dimension} of a finite positive Borel measure $\nu$ at a point $x$ is defined by
\be \label{locdim1}
{d}^-(\nu,x) = \liminf_{r\to 0} \frac{\log \nu(B_r(x))}{\log r}.
\ee
Equivalently,
\be \label{locdim2}
{d}^-(\nu,x) = \sup\bigl\{\gam \ge 0:\ \nu(B_r(x)) \le Cr^\gam,\ \mbox{for all}\ r>0,\ \mbox{for some}\ C>0\bigr\}.
\ee
Note that the definition is not changed if we only require the upper bound for the measure of balls of sufficiently small radius, since the measure is finite.
{ We will say that a measure $\nu$ is {\em H\"older regular} on a set $F$ if
\be \label{Ho-reg}
\exists\,\alpha>0\ \ \mbox{such that}\ \ {d}^-(\nu,x)\ge \alpha\ \ \mbox{for all}\ x\in F.
\ee
}

{ Let $\fra$ be an elementary admissible deformation of $X_\om$. Recall that the deformation was well-defined only on a {\em transversal} of the tiling space, namely, on the tilings
having a vertex at the origin, whereas the deformed tiling space $X_\om^\fra$
is simply the translation orbit closure of one representative deformed tiling. For concreteness, we fixed a vertex $v\in \Tk$,
an L-PSS tiling in $X_\om$, and define the deformed prototiles and super-prototiles for $j\le m$ by 
$$
T_j^\fra:= (T_j - v)^\fra\in (\Tk-v)^\fra;\ \ \ \ T_j^{\fra,n} = (\om^n(T_j) - v)^\fra\subset (\Tk-v)^\fra,\ \ j\le m.
$$
see \eqref{def-prot}, \eqref{def-prot2}.

For a tiling $\Sc^\fra\in X_\om^\fra$ and $j\le m$, let $\Lc_j(\Sc^\fra)$ be the Delone set of translation vectors between the prototiles $T_j^\fra$ and the tiles in $\Sc^\fra$ equivalent to it,
that is,
\be \label{dual2}
\Sc^\fra = \bigcup_{j\le m} \bigl(T_j^\fra + \Lc_j(\Sc^\fra)\bigr).
\ee
}

Now we are ready to state our main result. 
Assume for simplicity that the test function $\phi$ is a TLC function of level 0 on $\X^\fra_\omega$, represented as
\be \label{TLC1}
\phi(\Sc^\fra) = \sum_{j=1}^m \sum_{x\in \Lc_j(\Sc^\fra)} \delta_x * \psi_j(0),
\ee
{where $\psi_j$ is an integrable function, with 
$\supp(\psi_j) \subset T_j^\fra$.} General TLC functions may be represented in a similar way, using higher-level supertiles. 

\begin{theorem} \label{thm1} Let $\Tk$ be an FLC primitive aperiodic L-PSS tiling and $X_\om$ the corresponding tiling space. 
Let $\Gam$ be the $\Z$-module generated by elementary recurrences for $\Tk$ 
and $\{a_1,\ldots,a_s\}$ a set of free generators for $\Gam$. 
Let $\fra$ be a shape
parameter defining an admissible elementary deformation,
 ${\bf L}_\fra = [\fra(a_1),\ldots,\fra(a_s)]$, and 
 $\X^\fra_\omega$ the deformed substitution tiling space. Consider the measure-preserving system $(X^\fra_\om, \R^d,\mu_\fra)$.
For a  TLC function $\phi$ on $\X^\fra_\omega$ of the form \eqref{TLC1}, let $\sig_\phi$ be the corresponding spectral measure on $\R^d$. Then
\be \label{dim1}
{d}^-(\sig_\phi,\lamb) \ge 2\min\left\{ d - \frac{\chi^+(\bz,\vec\zeta)}{\log\theta},1\right\},\ \ \lamb \in \R^d\setminus \{0\},
\ee
where $\bz = {\bf L}_\fra^{\sf T}\lamb$ and $\vec\zeta = [\what\psi_1(\lamb),\ldots,\what\psi_m(\lamb)]^{\sf T}$.

 For $\lamb=0$, suppose that $\phi$ is orthogonal to constants, i.e., $\int_{X_\om^\fra} \phi\,d\mu_\fra = 0$. Then
\be \label{dim101}
{d}^-(\sig_\phi,0) \ge 2\min\left\{ d - \frac{\log|\vartheta_2|}{\log\theta},1\right\},
\ee
where $\vartheta_2$ is the second eigenvalue of $\Sf_\om$ (the PF eigenvalue being $\vartheta_1 = \theta^d$).
\end{theorem}

\begin{remark} \label{rem-thm1} {\em
(a) This is an extension to higher dimensions of the lower bound in \cite[Theorem 4.6]{BuSo20a} (in the case of a single substitution); it is essentially contained in \cite{Tre20}, although it is not stated there in terms of the spectral cocycle.

(b) The reason for the ``switch'' in the estimate at $\chi^+(\bz) =  (d-1)\log\theta$ is due to ``boundary effects,'' as in \cite{BuSo13,ST19,Tre20}.

(c) For unperturbed tiling self-similar spaces, there are more precise, two-sided bounds for the local dimension of the spectral measure at zero and even asymptotic expansions; see \cite{BuSo13} (for one-dimensional tilings) and \cite{Emme} (for $d>1$).

(d) Juan Marshall Maldonaldo \cite[Theorem 3.10]{Juan} proved that for a self-similar tiling in $\R^d$, with an expansion diagonalizable over $\R$ that is strongly non-Pisot (i.e., has an eigenvalue outside the unit circle), spectral
measures are log-H\"older regular. This means that they satisfy an estimate of the form $\sig_f(B_r(\lamb)) \le C_{\lamb} \log(1/r)^\gam$, $r>0$, with $\gam>0$ independent of $\lamb\ne 0$. This is an extension of \cite[Theorem 5.1]{BuSo14}
which obtained the result for $d=1$.

(e) Using the Fourier matrix cocycle \eqref{Four-coc}, 
Baake, Grimm, and Ma\~nibo \cite[Theorem 5.7]{BGM} proved that for an unperturbed self-similar tiling (even for self-affine and even without the FLC
assumption), if there exists $\eps>0$ such that $\chi^+(\lamb) \le \frac{d}{2}\log\theta-\eps$ for Lebesgue-a.e.\ $\lamb$, then the diffraction spectrum (which is essentially a ``part'' of the
dynamical spectrum) is purely singular. For $d=1$ this follows from \cite[Cor.\,4.7]{BuSo20a}, but for $d\ge 2$ we cannot make such a conclusion by our methods, essentially because
$(d-1)\log\theta\ge \frac{d}{2}\log\theta$.

(f) { It would be interesting to also obtain {\em upper bounds} for the local dimension of spectral measures, by analogy with \cite{BuSo20a}. We have some partial results in this direction and
hope to return to this question in the future.}
}
\end{remark}


\section{Eigenvalues,  quantitative Host-Veech criterion, and H\"older regularity of spectral measures} \label{sec:eig}

The material in this section is not particularly new, but included for completeness, { and we also emphasize the connections}.

\subsection{Eigenvalues}
Recall that $\lamb\in \R^d$ is a {\em topological eigenvalue} for the tiling dynamical system $(\X^\fra_\omega,\R^d)$ if there exists a continuous function $\phi: \X^\fra_\omega\to \C$ such that
\be \label{eigen1}
\phi(\Sc^\fra - \bz) = \exp\bigl(-2\pi  i\langle \lamb,\bz\rangle\bigr) \cdot\phi(\Sc^\fra)\ \ \mbox{for all}\ \ \Sc^\fra \in \X^\fra_\omega,\ \bz\in \R^d.
\ee

\begin{theorem}[{variant of \cite[Theorem 4.1]{CS06}}] \label{th-CS} 
Let $\Tk$ be an FLC primitive aperiodic L-PSS tiling and $X_\om$ the corresponding tiling space. 
Let $\fra$ be a shape
parameter defining an elementary admissible deformation,
 $\X^\fra_\omega$ the deformed substitution tiling space, and ${\bf L}_\fra$ given by \eqref{def-Lf}.
A  vector $\lamb\in \R^d$ is in the topological point spectrum of $(\X^\fra_\omega,\R^d)$ if and only if for every elementary recurrence vector $\bv$ of $\Tk$,
\be \label{eigen-cond1}
\langle \lamb,{\bf L}_\fra M^n \bv\rangle \to 0\ \mbox{(mod 1)},\ \ n\to \infty,
\ee
where the convergence is exponentially fast and uniform in $\bv$. Here $M$ is the matrix from \eqref{lin-map}.
\end{theorem}

{ In fact, our setting is more general, since \cite{CS06} assumed a pure dilation expansion map, whereas we allow an expansion which is a general similitude. However, the proof transfers
almost verbatim. Furthermore,
in \cite[Theorem 4.1]{CS06} only uniform in $\bv$ convergence is claimed, but the exponential rate follows from the proof immediately. }

Next we show that the same condition characterizes {\em measurable} eigenvalues for the uniquely ergodic system $(\X^\fra_\omega,\R^d,\mu)$. Thus, for admissible deformations of 
pseudo-self-similar
tiling spaces weak mixing is equivalent to topological weak mixing. 
These results follow a long line of earlier work, starting with Host \cite{Host86}, who obtained a similar criterion for the eigenvalues of (symbolic) substitution $\Z$-actions, in terms of return words, and also proved that every measure-theoretic eigenvalue is topological. For interval exchange transformations and translation flows an analogous 
criterion was obtained by Veech \cite{Veech}.

\begin{prop} \label{prop-eigen} Under the assumptions of Theorem~\ref{th-CS},
a  vector $\lamb\in \R^d$ is in the topological point spectrum of $(\X^\fra_\omega,\R^d)$ if and only if it is in the point spectrum of the uniquely ergodic measure-preserving system
$(\X^\fra_\omega,\R^d,\mu_\fra)$.
\end{prop}

The proof is a minor variation of the argument in \cite{SolTil}; we sketch it for completeness in Section~\ref{sec:proofs}.
Note that
\be \label{inprod}
\langle \lamb, {\bf L}_\fra M^n \bv\rangle = \langle  {\bf L}^{\sf T}_\fra\lamb, M^n \bv\rangle_{_{\R^s}} = \bigl\langle {(M^{\sf T})}^n({\bf L}^{\sf T}_\fra\lamb), \bv\bigr\rangle_{_{\R^s}} ,
\ee
where $\langle \cdot,\cdot\rangle$ is the inner product in $\R^d$, as opposed to the inner product in $\R^s$ in the right-hand side.

\smallskip

\begin{lemma} \label{cos-est}
There exists a uniform $c_1>0$ with the following property.
Let $\bv$ be a recurrence vector for $\Tk$, such that for all $j\le m$ there exists $k$ satisfying
\begin{equation} \label{conduc}
\exists\, z_1, z_2 \in \Dk_{jk}\ \ \mbox{such that}\ \ \bv = \alpha(z_1,z_2).
 \end{equation}
Then
\be \label{estap}
\|\Mc({\bf L}^{\sf T}_\fra\lamb,n)\|\le \const\cdot \theta^{nd} \prod_{i=0}^{n-1}
\left(1 - c_1 \bigl\|\bigl\langle {(M^{\sf T})}^i ({\bf L}^{\sf T}_\fra\lamb), \bv\bigr\rangle_{_{\R^s}}\bigr\|_{_{\R/\Z}}^2\right),\ \ n\ge 1.
\ee
\end{lemma}

For the proof, see \cite{Tre20}; it is a generalization of \cite{BuSo14} to higher rank actions. \qed

\medskip

In view of primitivity and repetitivity, for any given recurrence vector $\bv$, the property \eqref{conduc} holds if we replace $\omega$ by a sufficiently hight power. 
Passing from $\omega$ to  $\omega^k$ for some $k\in \N$, we can assume without loss of generality that \eqref{conduc} holds for a set of recurrence vectors $\bv$ generating $\Z^s$.

The following lemma is elementary, see e.g., \cite[Lemma 5.1]{BuSo21}.

\begin{lemma} \label{lem-lattice}
Let $\Z^s = \Z[\bv_1,\ldots,\bv_\ell]$ for some $\ell\ge s$. Then we have for $\bx\in \R^s$:
$$
\max_{j\le \ell} \|\langle \bx,\bv_j\rangle\|_{\R/\Z} \asymp \|\bx\|_{\R^s/\Z^s},
$$
with implied constants depending only on the generating set $\bv_1,\ldots,\bv_\ell$.
\end{lemma}

\begin{corollary} \label{cor:eigen}
 Under the assumptions of Theorem~\ref{th-CS}, there exists $k=k(\om)\in \N$ such that 
 
{\bf (i)} $\lamb\in \R^d$ is in the discrete spectrum of the system $(\X^\fra_\omega,\R^d)$ (in the topological or measurable category) if and only if
$$
\lim_{i\to \infty} {\bigl\|(M^{\sf T})^{ki} ({\bf L}^{\sf T}_\fra\lamb)\bigr\|}_{\R^s/Z^s} = 0,
$$
and the convergence is exponential.

{\bf (ii)}
$$
\|\Mc({\bf L}^{\sf T}_\fra\lamb,n)\|
\le \const\cdot \theta^{nd} \prod_{i=0}^{n-1}\left(1 - \wt c_1  {\bigl\|(M^{\sf T})^{ki} ({\bf L}^{\sf T}_\fra\lamb)\bigr\|}^2_{\R^s/\Z^s}\right), \ \ n\ge 1.
$$
\end{corollary}

\begin{proof}
(i) follows from Theorem~\ref{th-CS} and Lemma~\ref{lem-lattice}.

(ii) is immediate from \eqref{conduc} and Lemmas \ref{cos-est} and \ref{lem-lattice}.
\end{proof}

\begin{sloppypar}
Thus, weak mixing of the tiling dynamical system is equivalent to 
$$
\forall\ \lamb\in \R^d \setminus \{0\}, \ \|(M^{\sf T})^{ki} ({\bf L}^{\sf T}_\fra\lamb)\|_{\R^s/{\Z^s}}  \not\to 0,\ \ i\to\infty,\ \ \mbox{for}\ k=k(\om).
$$
(this is analogous to the Veech criterion \cite{Veech}). On the other hand, if $\|(M^{\sf T})^{ki} ({\bf L}^{\sf T}_\fra\lamb)\|_{\R^s/{\Z^s}}  \not\to 0$ in some ``quantitative way'' (say, with a positive frequency the distance is greater than
$\delta$) for all $\lamb\in \R^d \setminus \{0\}$, then we get  quantitative weak mixing.
\end{sloppypar}


\subsection{H\"older regularity}

\begin{theorem}  \label{th:quanti}
Let $\Tk$ be an FLC primitive aperiodic L-PSS tiling and $X_\om$ the corresponding tiling space. 
Let $\Gam$ be the $\Z$-module generated by elementary recurrences for $\Tk$
and $M$ be the integer matrix of the induced action of the expansion map on $\Gam$. 
Let $\Mk$ be an open set of elementary admissible deformations. 
If the dimension of the (strictly) expanding subspace of $M$ is at least $d+1$, then for Lebesgue-a.e.\ $[\fra]\in \Mk$, {there exist a representative $\fra$ of $[\fra]$ such that} the deformed $\R^d$ action on $X_\om^\fra$ has
uniformly H\"older-regular spectral measures. 
More precisely, there exists $\alpha >0$, depending only on the tiling space and the substitution, such that for a.e.\ $[\fra]\in \Mk$ {\ with admissible representative $\fra$}
and every transversally locally constant function $\phi$ on $X_\om^\fra$, the spectral measure $\sig_\phi$, associated with the uniquely ergodic dynamical system
$(X_\om^\fra,\R^d,\mu_\fra)$, satisfies $d^-(\sig_\phi,\lamb) \ge \alpha$ for all
$\lamb\in \R^d\setminus \{0\}$.
\end{theorem}

{ This is an extension of \cite[Cor.\ 1]{Tre20} to the case of pseudo self-similar tilings with a general (not necessarily pure dilation) expansion map.
Analogously to \cite[Th.\ 2]{Tre20}, one can also deduce {\em uniform} rates of weak mixing for functions with sufficient regularity in the leaf direction, as well as bounds on integrals of correlations.
}

\smallskip

In many cases Theorem~\ref{th:quanti}
 may be applied based only on the knowledge of the algebraic properties of the expansion $\varphi$. It is well-known that all the eigenvalues of $\varphi$
must be (real or complex) algebraic integers, see e.g., \cite[Cor.\,4.2]{LeeSol08}.
 A family of algebraic integers $\Theta=\{\theta_1,\ldots,\theta_d\}$, all of absolute value greater than one, is called a {\em Pisot family} if every Galois conjugate of every $\theta_j\in \Theta$ is either another element of $\Theta$ or lies inside the unit circle. For $d=1$ this is the definition of a (real) Pisot number, and for $d=2$, with $\theta_1,\theta_2$ complex conjugates, this is the definition of a complex Pisot number. Under very general conditions, if the set of eigenvalues of $\varphi$ is a Pisot family, then the tiling dynamical
system is not weakly mixing, see \cite{LS2}, and sometimes even pure discrete. We will say that $\Theta$ is {\em strongly non-Pisot} if there exists $\theta_j\in\Theta$ such that 
at least one of its Galois conjugates has absolute value strictly greater than one.

\begin{corollary} \label{cor:quanti}
Under the assumptions of Theorem~\ref{th:quanti}, suppose that
 the set of eigenvalues of $\varphi$ is strongly non-Pisot. Then for Lebesgue-a.e.\ admissible
deformation $[\fra]\in \Mk$ the deformed $\R^d$ action on $X^\fra$ has uniformly H\"older-regular spectral measures.
\end{corollary}

\begin{proof}[Derivation of Corollary~\ref{cor:quanti} assuming Theorem~\ref{th:quanti}]
Let $\Gam$ be the $\Z$-module generated by elementary recurrences for $\Tk$. The expansion $\varphi$ induces an endomorphism of $\Gam$ given by the integer matrix $M$.
It follows that all the eigenvalues of $\varphi$ are algebraic integers, and moreover, every eigenvalue of $\varphi$ is also an eigenvalue of $M$, see e.g., \cite[Lemma 1.4.5]{SolDel}. 
But $M$ is an integer matrix, hence every
Galois conjugate of its every eigenvalue is also an eigenvalue of $M$. By the strongly non-Pisot assumption, the dimension of the expanding subpace for $M$ is at least $d+1$, and the claim follows
from Theorem~\ref{th:quanti}.
\end{proof}

To conclude this section, we want to show that the above results can be extended to general PSS tilings, not necessarily the ``special'' L-PSS ones, and we summarize now how this is done.

Let $\mathcal{T}$ be a PSS tiling and $X$ its tiling space. By Remark \ref{rem-pss} (c), there is a tiling space $X'$ and a tiling $\mathcal{T}'\in X'$ which is L-PSS and MLD equivalent to $\mathcal{T}$. Denote by $\Phi:X\rightarrow X'$ the homeomorphism of tiling spaces defined by the MLD equivalence. Let $\mathcal{M}':= \mathcal{M}(X')\subset H^1(X';\mathbb{R}^d)$ be the set of non-singular classes of deformations, as defined in \S \ref{subsec:deform}, and $\mathcal{M}:= \Phi^*\mathcal{M}'\subset H^1(X;\mathbb{R}^d)$ its pullback under the MLD map.

Consider a class $[\frak{f}]'\in \mathcal{M}'\subset H^1(X';\mathbb{R}^d)$ with representative $\mathfrak{f}$, a $\mathbb{R}^d$-valued, $\mathcal{T}'$-equivariant, {\ admissible} smooth 1-form. First, we claim that $\mathfrak{f}$ is also $\mathcal{T}$-equivariant. Indeed, the value of $\mathfrak{f}(x)$ depends on $\mathcal{O}^+_{\mathcal{T}'}(B_{R'}(x))$ for some $R'>0$. Moreover, by MLD equivalence, the patch $\mathcal{O}^+_{\mathcal{T}'}(B_{R'}(x))$ is determined by the patch $\mathcal{O}^+_{\mathcal{T}}(B_{R}(x))$ for some $R>0$. Thus whenever $\mathcal{O}^+_{\mathcal{T}}(B_{R}(x))=\mathcal{O}^+_{\mathcal{T}}(B_{R}(y))+x-y$ we have that $\mathfrak{f}(x) = \mathfrak{f}(y)$, that is, $\mathfrak{f}$ is $\mathcal{T}$-equivariant. Thus $\mathfrak{f}$ has a class in $H^1(X;\mathbb{R}^d)$, which is denoted by $[\mathfrak{f}]$, and it is indeed the image of $[\mathfrak{f}]'$ under the pullback map: $[\mathfrak{f}] = \Phi^*[\mathfrak{f}]'$.

Since $[\mathfrak{f}]'\in \mathcal{M}'$, by the results of \S \ref{subsec:deform}, the representative $\mathfrak{f}$ defines a homeomorphism $H_{\mathfrak{f}}:\mathbb{R}^d\rightarrow \mathbb{R}^d$ through $H_{\mathfrak{f}}(x) = \int_0^x\mathfrak{f}$, which satisfies $H_{\mathfrak{f}}(\mathcal{T}-t) = \mathcal{T}-H_{\mathfrak{f}}(t)$. Let $\mathcal{T}^{\mathfrak{f}'} := H_{\mathfrak{f}}(\mathcal{T}')$ and $\mathcal{T}^{\mathfrak{f}} := H_{\mathfrak{f}}(\mathcal{T})$ be the deformed tilings, which are both FLC and repetitive. As such, they have well-defined tiling spaces $X^{\mathfrak{f}'}$ and $X^{\mathfrak{f}}$, respectively. Consider the map $\Phi_{\mathfrak{f}}:= H_{\mathfrak{f}}\circ \Phi \circ H_{\mathfrak{f}}^{-1}$ sending $\mathcal{T}^{\mathfrak{f}}$ to $\mathcal{T}^{\mathfrak{f}'}$. This map extends by minimality to all of $X^{\mathfrak{f}}$, and gives a map $\Phi_{\mathfrak{f}}:X^{\mathfrak{f}}\rightarrow X^{\mathfrak{f}'}$. This map is an MLD equivalence: indeed, since $\mathfrak{f}$ is $C^\infty$ and bounded, the distorsion of the map $H_{\mathfrak{f}}$ (and that of its inverse) is bounded over sets with diameter less than some finite fixed diameter. As such, the tile(s) covering $x\in\mathbb{R}^d$ in $\mathcal{T}^{\mathfrak{f}'}$ is determined by the patch $\mathcal{O}^-_{\mathcal{T}^{\mathfrak{f}}}(B_R(x))$ for some $R$ large enough.

Finally, note that a function $\phi:X^{\mathfrak{f}}\rightarrow \mathbb{R}$ is TLC if and only if $\phi = \Phi_{\mathfrak{f}}^*\phi'$ for some TLC function $\phi':X^{\mathfrak{f}'}\rightarrow \mathbb{R}$. As such, we have that $S_R^y(\phi,\lamb) = S_R^{\Phi_{\mathfrak{f}}(y)}(\phi',\lamb)$ for any $\lamb\in\mathbb{R}^d$, where $S_R^y(\phi,\lamb)$ is the twisted ergodic integral of $\phi$, introduced in \S \ref{subsec:twisted}, and used in all proofs of bounds of lower local dimension. As such, bounds on $S_R^y(\phi,\lamb)$ are equivalent to bounds on $S_R^{\Phi_{\mathfrak{f}}(y)}(\phi',\lamb)$ (see Lemma \ref{lem-spec1}). Thus we have proved the following.
\begin{corollary}
  Theorem \ref{th:quanti} holds with the weaker assumption of $\mathcal{T}$ being pseudo self-similar, not necessarily L-PSS. Moreover, if the set of eigenvalues of the inflation of a PSS tiling is strongly non-Pisot, then for Lebesgue-a.e.\ class $[\fra]\in \Mk$, {there exists a representative $\fra$ such that} the deformed $\R^d$ action on $X^\fra$ has uniformly H\"older-regular spectral measures.
\end{corollary}

\subsection{Quantitative Veech Criterion}
The proof of Theorem~\ref{th:quanti} is based on the following proposition, where we assume that $k=1$ without loss of generality.
This type of result goes by the name of {\em Quantitative Veech Criterion}.
      \begin{prop}
        \label{Veech:quant}
        Let $[\mathfrak{f}]\in\mathcal{M}$ and $\mathfrak{f}$ {an admissible} representative. If there exist $\rho,\delta\in(0,\frac{1}{2})$ such that 
                \begin{equation}
          \label{Veech:times}
          \limsup_{N\rightarrow\infty}\frac{1}{n}\left|\left\{n\in \mathbb{N}\cap (0,N): {\bigl\|(M^{\sf T})^{n} ({\bf L}^{\sf T}_\fra\lamb)\bigr\|}_{\R^s/\Z^s}<\rho\right\}\right|<1-\delta
        \end{equation}
        for $\lamb\in \R^d\setminus \{0\}$, then there exists $\alpha>0$, depending only on $\rho$ and $\delta$, such that for any transversally locally constant function $\phi:X^\mathfrak{f}\rightarrow \mathbb{R}$ holds $d^-(\sig_\phi,\lamb) \ge \alpha$.
      \end{prop}
      \begin{proof}
        This follows from Lemma~\ref{cos-est} and \eqref{dim1}. For details, see \cite[\S 5]{BuSo21} and \cite[Prop.\, 4]{Tre20}.
      \end{proof}
      
      { Recall that when we deform tilings and their spaces we do so through admissible {\em representatives} $\fra$ of classes in $\Mk$. If we pick a basis $\ba = \{a_1,\ldots,a_s\}$ of the $\Z$-module $\Gam$, then we obtain the shape vector ${\bf L}_\fra$ be evaluating $\fra$ on each element $a_i$ and obtaining the vector $\fra(a_i)\in\mathbb{R}^d$. The action of $\fra$ on this basis is by definition independent of representative $\fra$ of its class $[\fra]$ used and is given by the $d\times s$ matrix ${\bf L}_\fra$. Therefore the shape vector is assosciated to a class and not representative.  Below we write $\Mk_\ba$ for an open set of $d\times s$ matrices which are sufficiently small perturbations of $[\fra(a_1),\ldots,\fra(a_s)]$. In particular, all  matrices in $\Mk_\ba$ are uniformly bounded and have maximal rank $d$;
more precisely, there are $d$ columns such
that the determinant of the corresponding $d\times d$ matrix is bounded away from zero in modulus by a uniform $c>0$.
}      
      
\begin{proof}[Proof sketch of Theorem~\ref{th:quanti}] This is a brief sketch; for details the reader should consult \cite{BuSo21} and \cite{Tre20}.
Denote
$$
\eps_n(\bz) := {\|(M^{\sf T})^{n} \bz\|}_{\R^s/\Z^s},\ \ \bz\in \R^s.
$$
The exceptional set is defined as follows: for $B>1$ and $N\in \N$ let     
\begin{equation} \label{EK0}
\begin{split}
E_N(\rho,\delta,B) & = \left\{\bz\in \R^s: \ B^{-1} \le \|\bz\| \le B,\ \ \bigl|\{n\le N: \eps_n(\bz) \ge \rho\}\bigr| < \delta N \right\},\\
 \Ek_N(\rho,\delta,B) & = \left\{\fra\in { \Mk_\ba}:\ \exists\,\lamb\in \R^d,\ \ L_\fra^{\sf T} \lamb \in E_N(\rho,\delta,B) \right\},
\end{split}
\end{equation}
and
$$
\FrE(\rho,\delta,B):= \bigcap_{N_0=1}^\infty \bigcup_{N= N_0}^\infty \Ek_N(\rho,\delta,B).
$$
The theorem easily follows from Proposition~\ref{Veech:quant}, combined with the next lemma, in view of the fact that $\Mk_\ba$ is an open subset of $\R^{sd}$. \end{proof}

{
\begin{lemma} \label{lem:EK00}
Let $E^+$ be the (strictly) expanding subspace for the linear map $M^{\sf T}$ on $\R^s$.  There is $\rho>0$ such that given any $\eps_1>0$, for every $\delta>0$ sufficiently small,
for all $B>1$,
\begin{equation}\label{eq:dimh}
\dim_H(\FrE(\rho,\delta,B)) \le \gam:= sd - \dim E^+ + d + \eps_1.
\end{equation}
\end{lemma}

 \begin{proof} The proof is  a version of the ``Erd\H{o}s-Kahane argument'' 
 (or a kind of quantitative ``linear exclusion'' in the spirit of \cite[Section 7]{AF}).
 We provide a somewhat detailed sketch, since the proofs in \cite[\S 5]{BuSo21} and \cite[Prop.\,\,4]{Tre20} are rather technical, and our situation here is simpler: we apply a fixed
 transformation $M^{\sf T}$ instead of  a random sequence, so there is no need for Oseledets Theorem. Let
 \be \label{EK1}
 (M^{\sf T})^n\bz  = \Kb_n(\bz) + \beps_n(\bz),\ \ n\ge 0,
 \ee
 where $\Kb_n(\bz)\in \Z^s$ is (a) nearest lattice point to $(M^{\sf T})^n\bz$.  It will be convenient to use the $\ell^\infty$ norm, so that $\|\beps_n(\bz)\|_\infty=\eps_n(\bz) \le 1/2$.
 It follows from \eqref{EK1} that
 $$
 \Kb_{n+1}(\bz) + \beps_{n+1}(\bz) = M^{\sf T} \Kb_n(\bz) + M^{\sf T} \beps_n(\bz),\ \ n\ge 0,
 $$
 hence
 \be \label{EK2}
 \| \Kb_{n+1}(\bz) - M^{\sf T} \Kb_n(\bz)\|_\infty = \|M^{\sf T} \beps_n(\bz) - \beps_{n+1}(\bz)\|_\infty,\ \ n\ge 0.
 \ee
 Using that $M^{\sf T} \Kb_n(\bz)\in \Z^s$, we immediately obtain the following:
 
 \begin{lemma} \label{lem:EK}
 For all $n\ge 0$, independent of $\bz\in \R^s$, holds:
 
{\rm (i)} Given $\Kb_n(\bz)$, there are at most $(\|M^{\sf T}\|_\infty + 2)^d$ possibilities for $\Kb_{n+1}(\bz)$;
 
{\rm (ii)} Let
\be \label{def:rho}
\rho:= 2(\|M^{\sf T}\|_\infty +1)^{-1}.
\ee
If $\max\{\eps_n(\bz),\eps_{n+1}(\bz)\} < \rho$,
 then $\Kb_{n+1}(\bz) = M^{\sf T} \Kb_n(\bz)$.
 \end{lemma}
 
 We will use  $\rho$ from \eqref{def:rho} in Lemma~\ref{lem:EK00}.
  Lemma~\ref{lem:EK} implies that if $\delta >0$ is very small, then there is a good upper bound on the set of possible finite sequences 
  $\{\Kb_0(\bz),\ldots,\Kb_N(\bz)\}$ for $\bz\in E_N(\rho,\delta,B)$.
 We will next show that this yields an efficient covering of $E_N(\rho,\delta,B)$.

 Denote by $E^-$ the contracting subspace for $M^{\sf T}$ on $\R^s$, complementary to $E^+$. {Let $M^{\sf T}_+$ be the restriction of $M^{\sf T}$ to $E^+$}.
  Let $P_+$ and $P_-$ be the projections
 onto $E^+, E^-$ respectively, commuting with $M^{\sf T}$.
 From \eqref{EK1} we obtain, using that $P_+$ commutes with $M^{\sf T}$:
 $$
 P_+\bz = (M_+^{\sf T})^{-n} P_+(\Kb_n(\bz) + \beps_n(\bz)).
 $$
 Let $\kappa = \dim E^{+}$ and let $\theta_\kappa$ be the minimal  eigenvalue of $M$ { greater than 1}. Fix any $r\in (1,\theta_\kappa)$. Then
 $$
 \|(M_+^{\sf T})^{-n} P_+\| \le Cr^{-n},\ \ n\ge 0,
 $$
 and hence
 $$
 \|P_+\bz - (M_+^{\sf T})^{-n} P_+ \Kb_n (\bz)\| \le Cr^{-n},\ \ n\ge 0,
 $$
 where $C$ depends only on $M$. Thus, the knowledge of $\Kb_n(\bz)$ yields an approximation of order $\sim r^{-n}$ for $\P_\bz$.
 Now,  using the combinatorial counting argument from \cite[\S 5]{BuSo21}, we can conclude that the set $\{P_+\bz:\bz \in E_N(\rho,\delta,B)\}$ may be covered by 
 $O_{B,M}(1)\cdot \exp[L(\frac{1}{\delta} \log \delta) N]$ balls of radius $r^{-N}$, where $L$ is a uniform constant.  Choosing $\delta>0$ sufficiently small
 guarantees that this number is at most $O_{B,M}(1)\cdot r^{N\eps_1}$. This yields a covering of $E_N(\rho,\delta,B)$ by
 $O_{B,M}(1)\cdot r^{N(s-\kappa+\eps_1)}$ balls of radius $Cr^{-N}$, since $s-\kappa = \dim E^-$ and there are a priori bounds $\|z\|\le B$.
 
Finally, we need to pass from the covering of $E_N(\rho,\delta,B)\subset \R^s$ to a covering of $\Ek_N(\rho,\delta,B)\subset \Mk_\ba\subset \R^{ds}$. 
The latter is the set of matrices ${\bf L}_\fra$ such that ${\bf L}_\fra^{\sf T}\lamb=\bz\in E_N(\rho,\delta,B)$. The assumptions on $\Mk_\ba$ 
guarantees that $\lamb$ is bounded away from 0 and $\infty$ in norm. 

For $\bz\in \R^s$ the set of ${\bf L}_\fra$ such that ${\bf L}_\fra^{\sf T}\lamb=\bz$, with $\bz$ fixed, can be parameterized (at least, locally)
as follows. Choose $d$ rows of ${\bf L}_\fra$ with a determinant bounded away from zero in modulus. Solving the $d\times d$ system yields a unique solution $\lamb$. Then
the remaining $s-d$ rows in the linear system define a co-dimension 1 hyperplane each as the set of possibilities for the rows of ${\bf L}_\fra$, resulting in a set of total dimension
$d\times d + (s-d)\times (d-1) = sd - s + d$. Hence the set $\Ek_N(\rho,\delta,B)$ may be covered by 
$$
O_{B,M}(1)\cdot r^{N(sd+d-\kappa+\eps_1)} = O_{B,M}(1)\cdot r^\gam 
$$
balls of radius $Cr^{-N}$.
Now \eqref{eq:dimh} follows by a standard estimate of the Hausdorff dimension of limsup sets.
\end{proof}
}

{ 
\begin{remark} {\em
The last proof highlights the fact that the main ``play'' in the proof of H\"older regularity occurs in the strictly expanding subspace $E^+$ for $M^{\sf T}$. One may ask what is the role of the
contracting subspace. The following result of Clark and Sadun \cite{CS06} shows that by perturbing the deformation in the contracting direction we obtain a topologically conjugate system,
hence all the spectral properties remain unchanged.}
\end{remark}

\begin{prop}[{corollary of \cite[Theorems 2.2 and 3.1]{CS06}}] \label{prop:conj}
Let $\Tk$ be an FLC primitive aperiodic L-PSS tiling and $X_\om$ the corresponding tiling space. 
For two classes $[\fra], [\gra]\in \mathcal{M}\subset H^1(X_\om,\R^d)$ consider the deformed tiling spaces
$X_\om^\fra$ and $X_\om^\gra$, where $\fra,\gra$ are admissible representatives.
Let $E^-$ be the contracting subspace for $M^{\sf T}$, i.e., the span of (generalized) eigenvectors corresponding to eigenvalues less than one in modulus. If there exists $k\in \N$ such that
$(M^{\sf T})^k({\bf L}_\fra)-{\bf L}_\gra \in E^-$ or $(M^{\sf T})^k({\bf L}_\gra)-{\bf L}_\fra \in E^-$, then $(X^\fra_\om,\R^d)$ and $(X^\gra_\om,\R^d)$ are topologically conjugate.
\end{prop}

Although \cite{CS06} assumed a pure dilation expansion map, their proof extends verbatim.
}

{

\subsection{One-dimensional substitution tilings revisited} For one-dimensional substitution tilings, there is a space of ``natural'' deformations, obtained simply by changing the tile sizes.
Similarly to \cite[Section 5]{CS06}, we can extend  \cite[Theorem 4.1]{BuSo14}. First, we need to introduce some notation. Let $\zeta$ be a primitive aperiodic substitution on 
$d$ symbols, $(X_\zeta,T,\mu)$ the (2-sided) uniquely ergodic tiling $\Z$-action, and $(\Xxi^{\vec s}, h_t,\wt\mu)$ the suspension flow corresponding to a ``roof vector'' $\vec s\in \R^d_+$.
A word $v$ is called a return vector for $\zeta$ if $vc$ occurs in $x\in X_\zeta$ where $c$ is the 1st letter of $v$, that is, $v$ separates the two consecutive occurrences of $c$.
Denote by $\vec{\ell}(v)$ the ``population vector'' of a word $v$. Let $\Gam_\zeta<\Z^d$ be the $\Z$-module generated by $\{\vec{\ell}(v):\ v$ is a return vector for $\zeta\}$, and
the ``essential subspace'' $V_\zeta:=\Gam_\zeta \otimes \R$ the real linear span of $\Gam_\zeta$. First we recall the earlier result:

\begin{theorem}[{\cite[Th.\,4.1]{BuSo14}}]
Let $\zeta$ be a primitive aperiodic substitution on 
$d$ symbols with a substitution matrix $\Sf_\zeta$. Suppose that the characteristic polynomial of $\Sf_\zeta$ is irreducible over $\Q$ and the second eigenvalue satisfies $|\theta_2|>1$. Then for
Lebesgue-a.e.\ $\vec s\in \R^d_+$ the spectral measures of TLC functions for the system $(\Xxi^{\vec s}, h_t,\wt\mu)$ are H\"older regular away from zero, with a uniform H\"older exponent.
\end{theorem}

For comparison, now we have the following result, which is essentially a special case of Theorem~\ref{th:quanti}, with $d=1$.

\begin{theorem} \label{th-newa}
Let $\zeta$ be a primitive aperiodic substitution on 
$d$ symbols with a substitution matrix $\Sf_\zeta$ and the essential subspace $V_\zeta$. Suppose that the second eigenvalue of $\Sf_\zeta|_{V_\zeta}$ satisfies $|\theta_2|>1$. Then for
Lebesgue-a.e.\ $\vec s\in V_\zeta$ the spectral measures of TLC functions for the system $(\Xxi^{\vec s}, h_t,\wt\mu)$ are H\"older regular away from zero, with a uniform H\"older exponent.
\end{theorem}

\begin{remark} {\em 
(a) It is easy to see that if $\Sf_\zeta$ is irreducible, then $V_\zeta = \R^d$, however, the latter often happens even when $\Sf_\zeta$ is reducible.

(b) Suppose that $V_\zeta \ne \R^d$. Then it follows from \cite{CS03} that for $\vec s\in \R^d_+$ that is orthogonal to $V_\zeta$, the suspension flow $(\Xxi^{\vec s}, h_t,\wt\mu)$ has point spectrum
containing $\Z$. If $\vec s, \vec s'\in \R^d_+$ are such that $\vec s - \vec s'$ is orthogonal to $V_\zeta$, then the tiling spaces $\Xxi^{\vec s}$ and $\Xxi^{\vec s'}$ are MLD, hence the
flows are topologically conjugate.

(c) As it was pointed out in \cite[Section 5]{CS06}, suspension flows $(\Xxi^{\vec s}, h_t,\wt\mu)$ do not necessarily supply the entire space of admissible deformations; sometimes it happens that
collaring will result in a group (the analog of $\Gam_\zeta$) of higher rank than $d$.

}
\end{remark}
}
      

\section{Examples} \label{sec:ex}

\subsection{Kenyon's (pseudo) self-similar tilings}
In \cite{Kenyon}, given integers $p,q\geq 0$ and $r\in \mathbb{N}$, Kenyon introduced an algebraic construction of pseudo self-similar tilings using parallelograms, from which one can obtain a true self-similar tiling of $\mathbb{R}^2$. The nature of the construction is such that there is an obvious subspace $\mathcal{M}_{p,q,r}$ of deformation parameters which are accessible without having to compute the cohomology of the associated tiling spaces. In this section we give sufficient conditions under which a typical deformation of Kenyon's tilings, in the natural deformation space $\mathcal{M}_{p,q,r}$, we obtain tiling spaces which are quantitatively weak mixing.

Let us review the construction: let $a,b,c$ be three vectors in $\mathbb{R}^2$ pointing in different directions, and let $F$ be the set of polygonal paths starting at the origin, each of which is a translate of $\pm a,\pm b,$ or $\pm c$ without backtracking (i.e. $x$ is not followed by $-x$). A product can be defined on $F$ by concatenating paths and errasing any backtrack. As such any element in $F$ defines an element of the free group $F(a,b,c)$ on three generators and a natural isomorphism $h:F(a,b,c)\rightarrow F$.

Pick $p,q$ be non-negative integers, $r\in\mathbb{N}$, and let $\phi:F(a,b,c) \rightarrow F(a,b,c)$ be the endomorphism defined by
\begin{equation}
  \label{endo}
  \begin{split}
    \phi(a) &= b, \\
    \phi(b) &= c, \\
    \phi(c) &=c^pa^{-r}b^{-q}.
  \end{split}
\end{equation}
The three commutators $A = [a,b]$, $B = [b,c]$ and $C = [a,c]$ define three closed paths which enclose parallelograms which we label $A,B,C$. The action of the endomorphism on the commutators is thus
\begin{equation}
  \label{endoSub}
  \begin{split}
    \phi(A) &= B, \\
    \phi(B) &= c^pa^{-r}[a^r,c](b^{-q}[b^q,c]b^q)a^rc^{-q}, \\
    \phi(C) &=[b,c^p]c^pa^{-r}[a^r,b]a^rc^{-p}.
  \end{split}
\end{equation}
{ Consider the polynomial
\begin{equation}
  \label{polynomial}
  f(z) = z^3-p z^2+qz+ r.
\end{equation}
We have to impose the assumptions that $f$ is irreducible over $\Q$ and has a complex root $\lam$ of absolute value greater than 1: a complex Perron number, that is, a non-real
algebraic integer strictly greater in absolute value than its Galois conjugates other than $\ov\lam$.
This does not always happen: for example,  $f$ has a root $-1$, hence reducible, if $r=p+q+1$, and $z^3 - 4 z^2 + z + 1$ has three real zeros. Later we will also need the condition for all three roots to be outside of the unit circle, that is, for $\lam$ not to be a complex Pisot number.

\begin{lemma} \label{lem-Perron}
{\bf (i)} The polynomial $f$ has a complex root if and only if $p^2< 3q$, or $p^2\ge 3q$ and 
\begin{equation} \label{weird-cond}
27r > 2(p^2-3q)(p+ \sqrt{p^2-3q}) - 3pq.
\end{equation}

{\bf (ii)} If $f$ does have a complex root $\lam$, then it is complex Perron, unless $p=q=0$.

{\bf (iii)} If the above conditions are satisfied, all the roots are outside of the unit circle if and only if $r>p+q+1$. {If $r< p+q+1$, then $\lam$ is a complex Pisot number.}

{\bf (iv)} The polynomial $f$ is reducible over $\Q$ if and only if it has an integer root, which is necessarily a divisor of $r$.
\end{lemma}

\begin{proof}
(i) Note first that $df/dz = 3z^2 - 2pz + q$ has zeros $(p \pm \sqrt{p^2 - 3q})/3$,
hence if $p^2 < 3q$, we know that there is only one real root of $f$, hence there is a complex root.
If $p^2 \ge 3q$, then both (or the unique) extremal points $z_1 \le z_2$ are non-negative, and the condition for having a complex root is that $f(z_2) > 0$ (keeping in mind that $f(0) = r>0$).
A computation (left to the reader) shows that this is equivalent to \eqref{weird-cond}.

(ii) Suppose that the conditions from part (i) hold, and let $\lam$ be the complex root of $f$ with a positive imaginary part.
If $p=q=0$, then $\lam^3 = -r$, hence it is not a complex Perron number. Suppose that $\max\{p,q\}\ge 1$. 
There is one negative zero $-\alpha$ and two complex zeros $\lam$ and $\ov{\lam}$, such that
$|\lam|^2\cdot \alpha = r$. Observe that
$$
f(-r^{1/3}) = -r - pr^{2/3} - qr^{1/3} + r < 0,
$$
hence $\alpha < r^{1/3} \implies |\lam| > \alpha$, 
and so $\lam$ is complex Perron.

(iii) Assuming $\lam$ is complex Perron, all zeros are greater than one in absolute value if and only if $-\alpha < -1$, and this  is equivalent to $f(-1) = -1-p -q+r>0$, implying the first claim. 
{If $f(-1) < 0$, then $-\alpha \in (-1,0)$ and $\lam$ is complex Pisot.}

(iv) The (ir)reducibility claim is immediate, since $f$ is a monic polynomial.
\end{proof}
}

For the rest of the section we assume that $p,q,r$ are such that $f$ is irreducible over $\Q$ and $\lam$ is a complex Perron number. In the above construction let
 $a,b,c$ be $1,\lambda,\lambda^2\in\mathbb{C}$. 
We can now express the operation in terms of parallelograms obtained from (\ref{endoSub}). With slight abuse of notation,
\begin{equation}
  \label{KenyonSub}
  \begin{split}
    \om(A) &= B, \\
    \om(B) &= \left[\bigcup_{j=0}^{q-1} B-r-j\lambda+p\lambda^2\right]\cup  \left[\bigcup_{j=1}^r C-j+p\lambda^2\right] , \\
    \om(C) &=  \left[\bigcup_{j=1}^r A-j+p\lambda^2\right]\cup \left[\bigcup_{j=0}^{p-1} B-j\lambda^2+p\lambda^2\right].
  \end{split}
\end{equation}
This is not exactly a substitution rule, but a ``substitution with amalgamation''. We will show how this gives both pseudo self-similar tilings and self-similar tilings as a limit of the pseudo self-similar ones.

Using the rules $\om(K + x) = \om(K) + \lambda x$ and $\om(K_1\cup K_2) = \om(K_1)\cup \om(K_2)$ for $K\in \{A,B,C\}$, we can iterate the rules given in (\ref{KenyonSub}) to obtain larger and larger patches $K_n = \om^n(K)$ for every $n\in\mathbb{N}$. The patches $K_n$ grow in area exponentially with $n$ while containing the origin, and so in the limit they define tilings $\mathcal{T}_K$ of $\mathbb{R}^2$ belonging to the same tiling space $X_{p,q,r}$.

\begin{prop}
{\bf (i)} The space $X_{p,q,r}$ is the tiling space of a pseudo self-similar tiling given by the ``substitution with amalgamation'' rule (\ref{KenyonSub}). Moreover, it is an L-PSS tiling.

{\bf (ii)} This tiling space is FLC, repetitive, and aperiodic.
\end{prop}
\begin{proof}
(i) For $K\in\{A,B,C\}$ and $\mathscr{K}\in \{\mathscr{A},\mathscr{B}, \mathscr{C}\}$ define the rescaled tiles
$$\mathscr{K}_n =\lambda^{-n}K_n =\lambda^{-n}\om^n(K).$$
These satisfy
$$\om(\mathscr{K}_n) = \om(\lambda^{-n}\om^n(K_n)) = \lambda^{-n}\om^{n+1}(K) = \lambda^{-n}K_{n+1} = \lambda \mathscr{K}_{k+1}.$$
Taking the limit as $n\rightarrow\infty$, there is a convergence of tiles $\mathscr{K}_n\rightarrow \mathscr{K}$ (see, e.g., \cite[Lemma 7.7]{SolTil})
and we obtain an actual self-similar tiling \cite{Kenyon} defined by
\begin{equation}
  \label{KenyonSubSS}
  \begin{split}
    \lambda \mathscr{A} &= \mathscr{B}, \\
    \lambda \mathscr{B} &= \left[\bigcup_{j=0}^{q-1} \mathscr{B}-r-j\lambda+p\lambda^2\right]\cup  \left[\bigcup_{j=1}^r \mathscr{C}-j+p\lambda^2\right] , \\
    \lambda \mathscr{C} &=  \left[\bigcup_{j=1}^r \mathscr{A}-j+p\lambda^2\right]\cup \left[\bigcup_{j=0}^{p-1} \mathscr{B}-j\lambda^2+p\lambda^2\right].
  \end{split}
\end{equation}
Let $\wt{X}_{p,q,r}$ the tiling space associated to the primitive substitution (\ref{KenyonSubSS}). It follows by construction that $\mathcal{T}\in X_{p,q,r}$ if and only if there exist countable sets $\Lambda_A,\Lambda_B,\Lambda_C\subset \mathbb{R}^2$ such that
\begin{equation}
  \label{KenyonMLD}
  \mathcal{T} = (A+\Lambda_A)\cup(B+\Lambda_B)\cup(C+\Lambda_C)\hspace{.3in} \mbox{ and }\hspace{.3in} \mathscr{T} = (\mathscr{A}+\Lambda_A)\cup(\mathscr{B}+\Lambda_B)\cup(\mathscr{C}+\Lambda_C).
  \end{equation}
As such, $\mathcal{T}$ and $\mathscr{T}$ are MLD and therefore so are $\lambda \mathcal{T}$ and $\lambda\mathscr{T}$. Since $\mathscr{T}$ is locally derivable from $\lambda\mathscr{T}$, $\mathcal{T}$ is locally derivable from $\lambda\mathcal{T}$, so $\mathcal{T}$ is pseudo self-similar with expanding map $\lambda$.

(ii)  The substitution matrix from Kenyon's construction from the polynomial $x^3 - p x^2+qx + r = 0$, where $p,q\geq 0 $ and $r\in\mathbb{N}$, is obtained from (\ref{KenyonSubSS}) and is
      \begin{equation}
        \label{eqn:sub}
        \Sf_{p,q,r}=\left(\begin{array}{ccc} 0&0&r\\ 1&q&p \\ 0&r&0 \end{array}\right),
      \end{equation}
      which has characteristic polynomial $x^3 - qx^2 - pr x- r^2$, with the Perron-Frobenius eigenvalue equal to $|\lam|^2$. The FLC property is immediate by construction, and repetitivity follows from
      the fact that the substitution matrix $\Sf_{p,q,r}$ is primitive. To show aperiodicity, it is convenient to work with the self-similar tiling space $\wt{X}_{p,q,r}$. Note that if $0\ne x\in \R^2$ is a 
      period, then $\lam\cdot x$ is a period as well, whence there is a lattice of periods and all tile frequencies must be rational. However, the frequencies are given by the components of the Perron
      eigenvector of $\Sf_{p,q,r}$, which are irrational, since we assumed irreducibility of $f(z)$. This is a contradiction, and the proof is complete.
      
      Finally, observe that all the conditions of an L-PSS tiling are satisfied, see Definition~\ref{def-RL}.
\end{proof}

 \begin{prop}
        \label{NormalWeakMixing} Suppose that $f(z) = z^3 - pz^2+qz+r$ is irreducible over $\Q$ and has a complex zero $\lam$. Then 
        the dynamics on the tiling spaces $X_{p,q,r}$ and $\wt{X}_{p,q,r}$ are weakly mixing if and only if $r>p+q+1$.
      \end{prop}
      
       \begin{proof}
      By \cite[Theorem 5.1]{SolTil} and \cite{SolUCP}, the dynamics on the tiling space $\wt{X}_{p,q,r}$ are weakly mixing if and only if $\lambda$ is not a complex Pisot number. 
      Lemma~\ref{lem-Perron} says that this is equivalent (provided the other assumptions hold) to $r>p+q+1$.
       The result for $X_{p,q,r}$ then follows by MLD equivalence to $\wt{X}_{p,q,r}$.
     \end{proof}
     
  { In Figures 1-3  below we show several examples of Kenyon's  tilings. The colors (in the electronic version) correspond to distinct {\em collared} tiles.}

\begin{figure}[h]
  \centering
  \includegraphics[width = 3.8in]{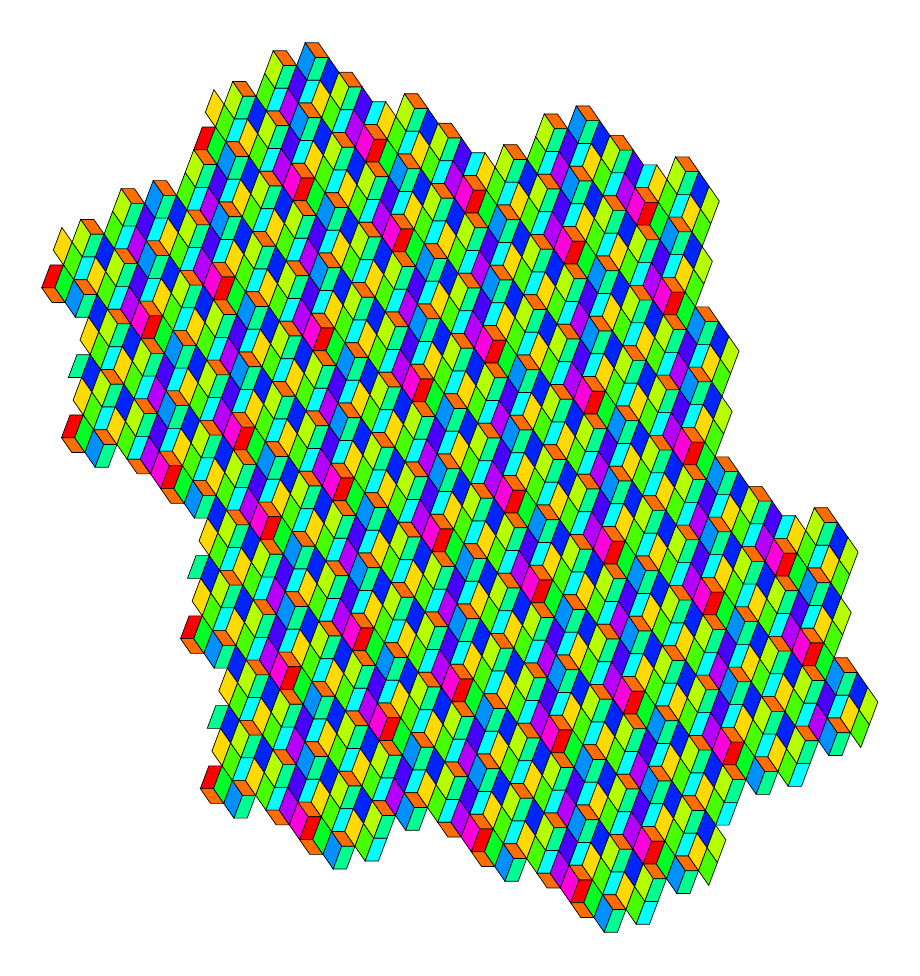}
  \caption{$(p,q,r)=(1,1,1)$, not weak mixing, 13 collared tiles, level 13 super-tile }
\end{figure}

\begin{figure}[t]
  \centering
  \includegraphics[width = 3in]{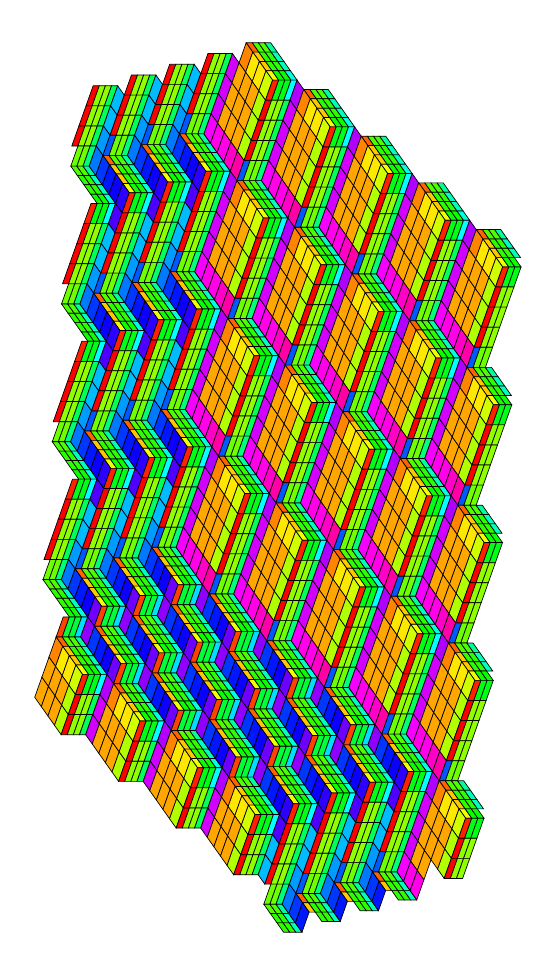}
  \caption{$(p,q,r)=(1,1,4)$,  weak mixing, 43 collared tiles, level 6 super-tile}
\end{figure}

\begin{figure}[t]
  \centering
  \includegraphics[width = 3.8in]{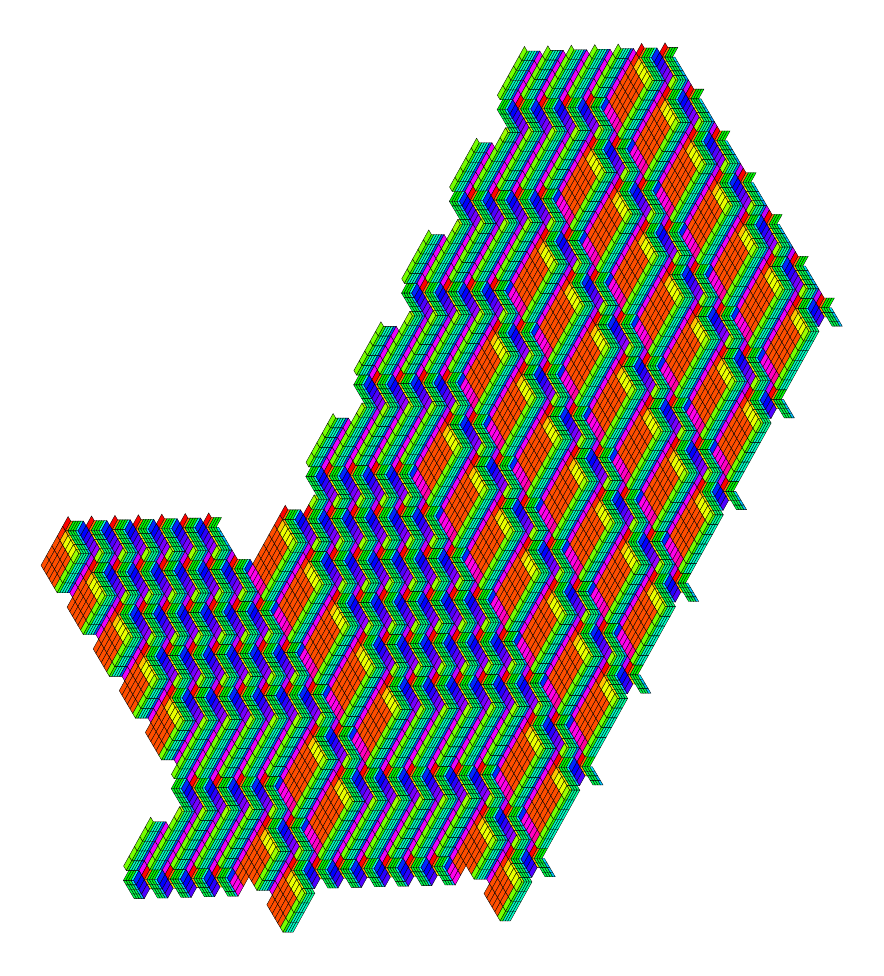}
  \caption{$(p,q,r)=(1,2,5)$,  weak mixing, 36 collared tiles, level 7 super-tile}
\end{figure}

   { \subsubsection{Deformations} 
Although \emph{all} deformations of the tiling space $X_{p,q,r}$ are given by an open subset of $H^1(X_{p,q,r};\mathbb{R}^2)$ (see Section \ref{sec:coh}), here we focus on 
 elementary deformations, corresponding to perturbations of the vectors $1,\lambda,\lambda^2\in\mathbb{C}$ which define the parallelograms. This saves us the the work of having to compute the entire space of deformations $H^1(X_{p,q,r};\mathbb{R}^2)$, which may or may not have higher dimension than 6.
This means that the role of the subgroup $\Gam < H_1(AP_0(X), \Z)$ in Section~\ref{sec-recur} will be played by the subgroup $\Z[1,\lam,\lam^2] < \C \cong \R^2$. }

From (\ref{KenyonSub}) and (\ref{KenyonSubSS}) it follows that the tiling spaces $X_{p,q,r}$ and $\wt{X}_{p,q,r}$ have the same group generated by the  return vectors $\Gamma_{p,q,r}$, and this group is necessarily a subgroup of $\Z[1,\lam,\lam^2]$.  In fact, we have

\begin{lemma}
$\Gamma_{p,q,r} = \mathbb{Z}[1,\lambda,\lambda^2]$.
\end{lemma}

\begin{proof}
By construction, $\Gamma_{p,q,r}$ is a subgroup of $\mathbb{Z}[1,\lambda,\lambda^2]$, invariant under multiplication by $\lam$. If $r\ge 2$, we immediately obtain that $1\in \Gamma_{p,q,r}$
from \eqref{KenyonSubSS}, since the substitution of $\Bc$ contains two translates of $\Cscr$ differing by 1, and the claim follows. If $r=1$ and $q\ge 2$, then $\lam\in \Gamma_{p,q,r}$ by
the formula for $\lam\Bc$, and then $\{\lam,\lam^2,\lam^3\}\subset \Gamma_{p,q,r}$. If $r=1$, then the algebraic number $\lam$ is a unit, and we  conclude that $1\in \Gamma_{p,q,r}$ by 
\eqref{polynomial}.
 A similar argument works for $r=1$ and $p\ge 2$, since then $\lam^2\in \Gamma_{p,q,r}$. The remaining cases are when $(p,q,r)\in \{(1,0,1),(0,1,1),(1,1,1)\}$, which are treated 
separately. For instance, if $(p,q,r)=(1,0,1)$ we obtain by iterating \eqref{KenyonSubSS} that $\lam^3 \Cscr$ contains $\Cscr-1+\lam^3$ and $\Cscr-1+\lam^2 + \lam^3$, hence $\lam^2\in \Gamma_{p,q,r}$ and we conclude as above. The remaining two cases are similar and are left to the  reader.
\end{proof}

{ To proceed, we need  the sets $\Dk_{ij}$ from the definition of self-similar tiling \eqref{subs1} to be subsets of $\Gam_{p,q,r}$,
and this holds in our case. In fact, identifying the prototile labels by $(\Ac,\Bc,\Cscr) \equiv (1,2,3)$, we obtain from \eqref{KenyonSubSS}:
\begin{equation*}
\begin{split}
\Dk_{12} = \{0\},\ \  \ \ \ & \Dk_{22} = \{-r-j\lam+p\lam^2\}_{j=0}^{q-1},\ \ \ \ \  \Dk_{23} = \{-j + p\lam^2\}_{j=1}^r,\\
&  \Dk_{31} = \{-j+ p\lam^2\}_{j=1}^r,\ \ \ \ \ \ \ \ \ \Dk_{32} = \{-j\lam^2 + p\lam^2\}_{j=0}^{p-1},
\end{split}
\end{equation*}
with all the remaining $\Dk_{ij}$ being empty.

Let $\alpha_{p,q,r}:\Z[1,\lam,\lam^2]=\Gam_{p,q,r}\rightarrow \mathbb{Z}^3$ be the address map defined by $\alpha_{p,q,r}(n_1+n_2\lambda+n_3\lambda^2) = (n_1,n_2,n_3)\in\mathbb{Z}^3$. 
The inverse map is explicitly given by
       $$\alpha^{-1}_{p,q,r}(n_1,n_2,n_3) = \left(\begin{array}{ccc}1&\Re(\lambda) & \Re(\lambda^2) \\ 0&\Im(\lambda) & \Im(\lambda^2)   \end{array}  \right)\left(\begin{array}{c}n_1 \\ n_2\\ n_3 \end{array}\right) = V_{p,q,r}\left(\begin{array}{c}n_1 \\ n_2\\ n_3 \end{array}\right)  = n_1+n_2\lambda + n_3\lambda^2,$$
        where $V_{p,q,r}$ is the matrix with column vectors $1,\lambda, \lambda^2\in\mathbb{C}$.

There is a neighborhood $\mathcal{M}_{p,q,r}$ of $(1,\lambda,\lambda^2)\in\mathbb{C}^3\cong \mathbb{R}^6$ which parametrizes non-degenerate deformations of parallelograms $A,B,C$. 
 In fact, one can check that the deformed tiling is well-defined if we let the vector $a$ point in the positive direction of the $x$-axis, and vectors
$b,c$ point into the 1st and 2nd quadrant respectively.
(To this configuration we can of course apply a $GL(2,\R)$ map.)

We denote the deformed parallelograms by $\{A^\frak{f},B^\frak{f},C^\frak{f}\}$, for $\frak{f}\in\mathcal{M}_{p,q,r}$ close enough to $(1,\lambda,\lambda^2)\in\mathbb{C}^3$. This in turn defines deformed patches $\{A_n^\frak{f},B_n^\frak{f},C_n^\frak{f} \}$ obtained by deforming the individual parallelograms in the patches $\{A_n,B_n,C_n\}$ for all $n\geq 0$. Given that
$$\mathcal{T}_K =\bigcup_{n\geq 0} K_n,$$
for $K_n\in \{A_n,B_n,C_n\}$, the deformed patches define a deformation of $\mathcal{T}_K$ by
$$\mathcal{T}_K^\frak{f}:=\bigcup_{n\geq 0} K_n^\frak{f},$$
for $K_n^\frak{f}\in \{A_n^\frak{f},B_n^\frak{f},C_n^\frak{f}\}$. We denote by $X^\frak{f}_{p,q,r}$ the tiling spaces for $\mathcal{T}^\frak{f}_K$ which are deformations of the tiling space $X_{p,q,r}$.

        To deform $X_{p,q,r}$ we deform $V_{p,q,r}$ as a natural subset of $\mathbb{C}^3\cong\mathbb{R}^6$. Let $\mathfrak{f}\in\mathcal{M}_{p,q,r}$ be close to $(1,\lambda,\lambda^2)$ and denote by $V_{p,q,r}^\frak{f}$ the matrix associated with this deformation. { (This is the matrix ${\bf L}_\fra$ in our case.)} In other words, since $V_{p,q,r}$ has columns $1,\lambda, \lambda^2$, the matrix $V_{p,q,r}^\mathfrak{f}$ has columns $v_1^\mathfrak{f}, v_2^\mathfrak{f}$ and $v_3^\mathfrak{f}$ {(these are the vectors $a,b,c$ mentioned above)},
        which are respectively close to $1,\lambda$ and $\lambda^2$. The new { group generated by the} set of return vectors is therefore
        $$\Gamma^\mathfrak{f}_{p,q,r} = V^\frak{f}_{p,q,r}\cdot\alpha_{p,q,r}(\Gamma_{p,q,r}).$$
     
     {
     \noindent The expansion map $\varphi$ for the PSS tiling is multiplication by $\lam$ on $\C$, which induces a linear map on $\Z[1,\lam,\lam^2]\cong\Z^3$, given by the matrix
      \begin{equation}
        \label{eqn:1D}
      M=  \mathcal{G}_{p,q,r} :=\left(\begin{array}{ccc} 0&1&0\\ 0&0&1 \\ -r&-q&p \end{array}\right)
      \end{equation}
      which has characteristic polynomial $f(z) = z^3 - pz^2+qz+r$.  
     
     In order to write down the spectral cocycle, we  recall \eqref{spec-mat} and Definition~\ref{def-cocy} to obtain, denoting $e(t) := \exp(-2\pi i t)$:
     $$
     \Mc(\bz) = \Bigl[\sum_{\bx\in \Dk_{jk}} e\bigl(\langle \bz,  \alpha(\bx)\rangle\bigr)\Bigr]_{j,k\le 3};\ \ 
     \Mc(\bz,n)= \Mc\bigl({(M^{\sf T})}^{n-1}\bz\bigr)\cdot\ldots\cdot \Mc(\bz),\ \ \bz\in \T^3,\ n\in \N. 
     $$
     For example, if we take $(p,q,r)=(1,1,1)$ for simplicity, then $\Dk_{12} = \{0\}$, $\Dk_{22} = \Dk_{23} =\Dk_{31} = \{-1+\lam^2\}$, $\Dk_{32} = \{\lam^2\}$, which yields
     $$
     \Mc(\bz) = \left[\begin{array}{ccc} 0 & 1 & 0 \\ 0 & e(-z_1+z_3) & e(-z_1 + z_3) \\ e(-z_1+z_3) & e(z_3) & 0 \end{array} \right].
     $$
     
     }

      We now wish to extend Proposition~\ref{NormalWeakMixing}  in terms of both deformations and the H\"older property for the corresponding spectral measures.

  {
      \begin{prop}
        \label{QuantitativeKenyon}
         Suppose that $f(z) = z^3 - pz^2+qz+r$ is irreducible over $\Q$ and has a complex zero $\lam$. If $r>p+q+1$, then 
         
         {\bf (i)} for every admissible deformation parameter $\fra\in \Mk_{p,q,r}$ outside of a set of  codimension 1, the uniquely ergodic dynamics on $X_{p,q,r}^\frak{f}$
         is weakly mixing;
         
         {\bf (ii)}   
       for Lebesgue almost every deformation parameter $\fra\in \Mk_{p,q,r}$,
        the spectral measures for TLC functions associated to the uniquely ergodic dynamics on $X_{p,q,r}^\frak{f}$ have positive local dimension.
      \end{prop}
 
 \begin{proof}[Proof sketch] 
 (i) We can apply Theorem~\ref{th-CS} and Proposition~\ref{prop-eigen}. For  recurrence vector $\bv$ in \eqref{eigen-cond1} we can take $\bbe_1=(1,0,0)^{\sf T}$, representing $1\in \Gam_{p,q,r}$ in the
 basis $\{1,\lam,\lam^2\}$. These results say that if $\lamb\in \R^2$ is an eigenvalue (topological or measure-theoretic), then $\langle  {\bf L}^{\sf T}_\fra\lamb, M^n \bbe_1\rangle\to 0$ (mod 1)
 as $n\to \infty$, where ${\bf L}^{\sf T} = V_{p,q,r}^\fra$ in our case. Since all the eigenvalues of $M$ are greater than one in modulus, a standard argument (see, e.g., \cite{CS06}) implies that 
 for $\lamb\ne 0$ this can happen only when $\langle  {\bf L}^{\sf T}_\fra\lamb, M^n \bbe_1\rangle=0$ for all $n$ sufficiently large. But then the columns of $V_{p,q,r}^\fra$ must be rationally 
 dependent, and such $\fra$ form a countable union of subspaces of codimension 1.
 
 (ii) This is a special case of Corollary~\ref{cor:quanti}, in view of Lemma~\ref{lem-Perron}.
 \end{proof}
 }
 
     \begin{remark} {\em 
  If $f(z) = z^3 - pz^2+qz+r$ is irreducible over $\Q$, has a complex zero $\lam$, and $r< p+q+1$, then $\lam$ is a complex Pisot number. Then
   it follows from \cite{CS06}, { see Proposition~\ref{prop:conj}}, that all admissible deformations of the tiling space result in a topologically conjugate system.}
  \end{remark} 
 
\subsection{Square tilings}

This class of examples is obtained when all tiles are unit cubes (located at the vertices of $\Z^d$), but with different labels, and the expansion map is a diagonal map $qI$, where $q\in \N$, $q\ge 2$.
This is essentially a $d$-dimensional generalization of symbolic substitutions of constant length.

\begin{example}\label{bad-substitution2}{\em
All the tiles have the unit square as its support and are distinguished only by the labels. Let $\Ak = \left\{\begin{tabular}{|c|} \hline 0 \\ \hline \end{tabular}\,,
\begin{tabular}{|c|} \hline 1 \\ \hline \end{tabular}\,,\begin{tabular}{|c|} \hline 2 \\ \hline \end{tabular}\right\}$; the expansion is pure dilation by a factor of 6:
 $$
 \begin{tabular}{|c|} \hline 0 \\ \hline \end{tabular}\  \to\  \begin{tabular}{|c|c|c|c|c|c|} \hline 2 & 2 & 2 & 2 & 2 & 2 \\ \hline
 2 & 2 & 2 & 2 & 2 & 2 \\ \hline
                                                                              1 & 0 & 1 & 0 & 0 & 1 \\ \hline
                                                                           1 & 0 & 0 & 1 & 0&  1 \\ \hline
                                                                          2 & 2 & 2 & 2 & 2 & 2 \\ \hline
 2 & 2 & 2 & 2 & 2 & 2 \\ \hline
                                                                                                                           \end{tabular}\,, \ \ \ \ \ \ 
                                                                                                                                                                                                                            \begin{tabular}{|c|} \hline 1 \\ \hline \end{tabular}\  \to\                                         \begin{tabular}{|c|c|c|c|c|c|} \hline 2 & 2 & 2 & 2 & 2 & 2 \\ \hline
 2 & 2 & 2 & 2 & 2 & 2 \\ \hline
                                                                              1 & 0 & 1 & 0 & 0 & 1 \\ \hline
                                                                           1 & 0 & 1 & 1 & 0 &  1 \\ \hline
                                                                          2 & 2 & 2 & 2 & 2 & 2 \\ \hline
 2 & 2 & 2 & 2 & 2 & 2 \\ \hline
                                                                                                                           \end{tabular}\,, \ \ \ \ \ \
                                     \begin{tabular}{|c|} \hline 2 \\ \hline \end{tabular}\  \to\  \begin{tabular}{|c|c|c|c|c|c|} \hline 0 & 1 & 0 & 0 & 0 & 1\\ \hline
                                     0 & 0 & 1 & 1 & 0 & 0 \\ 
                                      \hline 2 & 2 & 2 & 2 & 2 & 2 \\ \hline
 2 & 2 & 2 & 2 & 2 & 2 \\ \hline
                                                                              1 & 0 & 0 & 1 & 0 & 0 \\ \hline
                                                                           1 & 0 & 0 & 1 & 0&  1 \\ \hline
                                                                                                                                                                                                     \end{tabular}
$$
}
\end{example}

\begin{figure}[t]
  \centering
  \includegraphics[width = 6.5in]{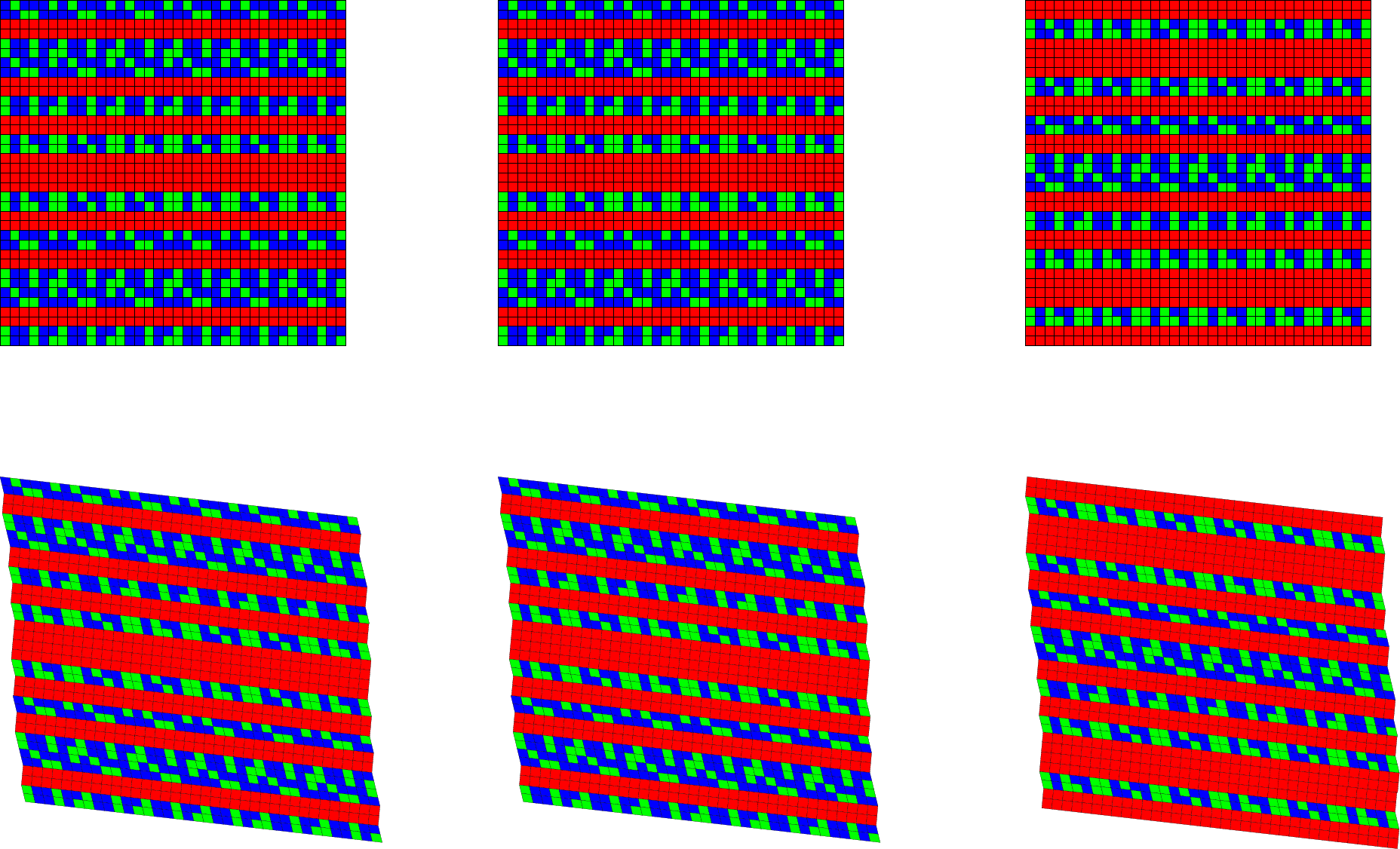}
  \caption{Square tiling and its deformation}
\end{figure}

\noindent
The substitution $\om$ was chosen in such a way that every tiling in the tiling space has two periodic rows, consisting entirely of 2's, followed by two rows consisting entirely of 0's and 1's, again followed
by two periodic rows, consisting entirely of 2's, etc., ad infinitum.
The substitution is clearly primitive. Although it has periodic rows, there are no global translational
 periods. This can be seen, e.g., by observing that the substitution is recognizable: the pattern 
 $\begin{tabular}{|c|c|} \hline1 &  0 \\ \hline  1 & 1 \\ \hline \end{tabular}$ can only occur in 
 $\om\left(\,\begin{tabular}{|c|} \hline 1 \\ \hline \end{tabular}\,\right)$ at the center of a super-tile. Moreover, it is easy to see that the substitution ``forces the border'' \cite{AP}.
 
We consider elementary deformations. 
Namely, we  deform the vectors-edges 
in a way consistent with the Anderson-Putnam complex ${AP}_{0}(X)$, so that the sum of the vectors around each tile is zero. 
It is easy to see that  all horizontal edges are identified in the AP-complex, but there are two kinds of vertical edges: those of the tiles labeled 2, and those of the tiles labeled 0 or 1. For each tile edge $e$, the deformation of the edge opposite to $e$ must be the opposite of the deformation of $e$; this will ensure consistency and all the deformed tiles will be parallelograms. 

On the other hand, working with recurrence will lead to the same conclusion. There are recurrences with return vectors $(1,0)$ and $(0,1)$. All ``horizontal'' recurrences {
project to a multiple of the same cycle}
, whereas the ``vertical ones" are of two types, similarly to the above. These recurrences form a basis for the lattice of all recurrences. It is not hard to compute the linear map induced
by the expansion --- it is a block matrix, with the blocks corresponding to ``side substitutions":
$$\left[\begin{array}{ccc} 2 & 4 & 0 \\ 4 & 2 & 0 \\ 0 & 0 & 6 \end{array} \right].$$ The eigenvalues are $6,6,2$, so now $\dim E^+ = 3$, and we  conclude that
a.e.\ deformation has quantitative weak mixing.


\section{Proofs} \label{sec:proofs}

\subsection{Spectral estimate}
\label{subsec:twisted}
The bounds for local dimension for spectral measures are based on growth estimates of  {\em twisted ergodic integrals}. The following lemma, which  essentially goes back to Hof \cite{Hof}, will be used in the proof of Theorem~\ref{thm1}.

\begin{lemma}[{\cite[Lem.\,1]{Tre20}}] \label{lem-spec1}
Let $(Y,\mu, h_t)_{t\in \R^d}$ be an ergodic probability-preserving $\R^d$ action and $\phi\in L^2(Y,\mu)$. For $\lamb\in \R^d$ and $y\in Y$ the twisted ergodic integral is 
$$
S_R^y(\phi,\lamb) := \int_{Q_R} e^{-2\pi i \langle \lamb, \tau\rangle} \phi\circ h_t(y)\,d\tau,
$$
where $Q_R = [-R,R]^d$. Suppose that for some $\lamb\in \R^d,\ R_0>0, \phi\in L^2(Y,\mu)$ and $\alpha \in (0,d)$,
$$
\|S_R^y(\phi,\lamb)\|_{L^2}  \le C_1 R^{d-\alpha}\ \ \mbox{for all}\ R\ge R_0.
$$
Then
$$
\sig_\phi(B_r(\lamb))\le C r^{2\alpha}\ \ \mbox{for all}\ r < 1/(2R_0),
$$
for some $C>0$. In particular, the lower local dimension of the spectral measure satisfies
$$
d^-(\sig_\phi,\lamb) \ge 2\alpha.
$$
\end{lemma}


\subsection{Twisted ergodic integrals over deformed super-tiles} {Recall that we are working with an L-PSS tiling space $X_\om$ with expansion map $\varphi$, 
where $\om$ is a
combinatorial substitution, or ``substitution-with-amalgamation,'' for which a pseudo-self-similar tiling $\Tk$ is a fixed point. The prototiles $\{T_j\}_{j\le m}$ are actual specific tiles of $\Tk$.
Let $\fra$ be an admissible deformation.
In order to consider deformed tiles and super-tiles, we fixed a vertex $v$ of $\Tk$ and defined 
$T_j^\fra:= (T_j - v)^\fra$ and $T_j^{\fra,n} = (\om^n(T_j) - v)^\fra$ for $j\le m$, see \eqref{def-prot} and \eqref{def-prot2}.
}

Below, when writing $\int_{T_j^{\fra,n}}$ and the like, we mean integration over the {\em support} of the corresponding patch.

Let $\phi$ is a TLC function of level 0 on $\X^\fra_\omega$ of the form \eqref{TLC1}. Let $j\le m$ and $\Tk^\fra\in X_\om^\fra$ be such that $(T_j^{\fra,n})\subset \Tk^\fra$ for all $n\ge 0$.
Then,  in view of \eqref{subs-f1}, 
\begin{eqnarray*}
\int_{T_j^{\fra,1}} e^{-2\pi i \langle \lamb,\bt\rangle}\phi(\Tk^\fra-\bt)\,d\bt  & = & \sum_{k=1}^m \sum_{\bx\in \fra(\Dk_{jk})} e^{-2\pi i \langle \lamb,\bx\rangle} \cdot \what \psi_k (\lamb) \\ 
& = & \sum_{k=1}^m \sum_{\bx\in \Dk_{jk}} e^{-2\pi i \langle {\bf L}_\fra^{\sf T}\lamb,\alpha(\bx)\rangle_{_{\R^s}}} \cdot \what \psi_k (\lamb) \\
& = & \bigl[\Mc({\bf L}_\fra^{\sf T}\lamb)\cdot \what\Psib(\lamb)\bigr](j) = \bigl[\Mc(\bz)\cdot \what\Psib(\lamb)\bigr](j),
\end{eqnarray*}
where
$$
\what\Psib(\lamb) = \left[\begin{array}{c} \what \psi_1 \\ \vdots \\ \what \psi_m \end{array} \right](\lamb).
$$
Iterating this, we obtain, in view of \eqref{subs-f2}, 
\begin{eqnarray*}
& & \int_{T_j^{\fra,2}} e^{-2\pi i \langle \lamb,\bt\rangle}\phi(\Tk^\fra-\bt)\,d\bt\\
 &  = & \sum_{s=1}^m \sum_{k=1}^m  \left(\sum_{\bx\in {\bf L_\fra} M \alpha(\Dk_{jk}) + {\bf L_\fra}\alpha(\Dk_{ks})} e^{-2\pi i \langle \lamb,\bx\rangle} \right) \what \psi_s(\lamb) \\
 & = & \sum_{k=1}^m \sum_{s=1}^m \left(\sum_{\bx\in \Dk_{jk}} e^{-2\pi i \langle M^{\sf T} {\bf L}_\fra^{\sf T}\lamb ,\alpha(\bx)\rangle_{_{\R^s}}} \cdot \sum_{\bx\in  \Dk_{ks}} e^{-2\pi i 
 \langle {\bf L}_\fra^{\sf T}\lamb ,\alpha(\bx)\rangle_{_{\R^s}}}\right)  \what \psi_s(\lamb),
\end{eqnarray*}
hence
$$
\int_{T_j^{\fra,2}} e^{-2\pi i \langle \lamb,\bt\rangle}\phi(\Tk^\fra-\bt)\,d\bt = \bigl[\Mc( M^{\sf T}\bz)\Mc(\bz)\cdot \what\Psib(\lamb)\bigr](j) = \bigl[\Mc(\bz,2)\cdot \what\Psib(\lamb)\bigr](j),
$$
and similarly, by induction, for all $n\ge 1$,
\be \label{itera}
\int_{T_j^{\fra,n}} e^{-2\pi i \langle \lamb,\bt\rangle}\phi(\Tk^\fra-\bt)\,d\bt = \bigl[\Mc(\bz,n)\cdot \what\Psib(\lamb)\bigr](j).
\ee
This immediately implies the following.

\begin{lemma} \label{lem-integ}
Suppose that $\Sk^\fra \in \X^\fra_\omega$ and $T_j^{\fra,n}+ \by$ is a super-tile of $\Sk^\fra$.
Then
\be \label{Lyapa1}
\int_{T^{\fra,n}_j+\by} e^{-2\pi i \langle \lamb,\bt\rangle}\phi(\Sk^\fra-\bt)\,d\bt= e^{-2\pi i \langle \lamb, \by\rangle} \cdot \bigl[\Mc(\bz,n)\cdot \what\Psib(\lamb)\bigr](j).
\ee
\end{lemma}

\subsection{Decomposition of patches}

The next lemma provides an efficient decomposition of a large patch of a tiling in $X_\om^\fra$ into the union of supertiles of different levels for an arbitrary Lipschitz domain $G$. 
It is analogous to the construction in \cite[Lemma 3.2]{BuSo13}. In the self-similar (undeformed), but non-stationary setting, it also appeared in 
\cite[Lemma 8.1]{ST19}, with $G = [-R,R]^d$.
Denote by $U(\partial G,r)$ the $r$-neighborhood of $\partial G$, and let $\Lk^{d-1}(\partial G)$ be the usual surface measure of the boundary (e.g., the Hausdorff measure).

\begin{lemma} \label{lem-decomp} Let $G$ be a Lipschitz domain in $\R^d$ such that $\Lk^d(U(\partial G,r)) \le C_G \cdot\Lk^{d-1}(\partial G)\cdot r$ for $r>0$.
Let $X_\omega$ be an L-PSS tiling space in $\R^d$ with expansion $\varphi$ and an elementary admissible deformation $\fra$, and let $X_\om^\fra$ be the deformed tiling space. 
Then for any $\Sk^\fra \in \X^\fra_\omega$  there exists an integer
$n= n(G)$ and a decomposition
\be \label{eq-decomp1}
\Ok_{\Sk^\fra}^- (G) = \bigcup_{i=0}^n \bigcup_{j=1}^m \bigcup_{k=1}^{\kappa^{(i)}_j}  T_{j,k}^{\fra,i},
\ee
where $T_{j,k}^{\fra,i}$ is a level-$i$ supertile of $\Sk^\fra$ of type $j$, such that
\begin{description}
\item[(i)] $\kappa_j^{(n)} \ne 0$ for some $j$;
\item[(ii)] $\sum_{j=1}^m {\kappa_j^{(i)}} \le C_\fra C_G \cdot\Lk^{d-1}(\partial G)\cdot \theta^{i-id}$ for $i=0,\ldots,n$.
\end{description}
Here $\theta = \|\varphi\|$ and  $C_\fra$  is a  generic constant  which depends only on the substitution rule.
\end{lemma}

\begin{proof} 
Recall that $\Ok_{\Sk^\fra}^- (G)$ denotes the patch of $\Sk^\fra$-tiles contained in $G$.
Consider $\Sk^\fra$ with the higher-order deformed super-tiles composed of its tiles. 
(We don't even need recognizability: it is always possible to ``de-substitute,'' by the definition of the tiling space.) Let $n$ be maximal such that an $n$-level $\fra$-deformed $\Sk^\fra$-super-tile
is contained in $G$, so that (i) is satisfied.
 Let $\Rk^{(n)}(G)$ be the set of all $n$-level $\fra$-deformed $\Sk^\fra$-super-tiles contained in $G$. Next, for $i=n-1,\ldots,0$, we inductively define $\Rk^{(i)}(G)$ to be the set of $i$-level $\fra$-deformed super-tiles, contained in $G$, but not contained in one of the super-tiles of the higher level in $G$ (that is, those which are contained in
$G \setminus \supp(\Rk^{(i+1)}(G))$). Let $\kappa_j^{(i)}$ be the number of super-tiles of level $i$ of type $j$ in $\Rk^{(i)}(G)$; these super-tiles are denoted $T_{j,k}^{\fra,i}$, for $k=1,\ldots,\kappa_j^{(i)}$.

By construction, the super-tiles of $\Rk^{(i)}(G)$ lie in super-tiles of level $(i+1)$ which {\em intersect} the boundary $\partial G$, hence
 by \eqref{roundish2}, all the super-tiles of $\Rk^{(i)}(G)$ are contained in the $C_\fra \cdot \theta^{i+1}$-neighborhood of $\partial G$. By assumption, the volume of this neighborhood is at most
 $C_G\cdot\Lk^{d-1}(\partial G)\cdot \theta^{i+1}$. On the other hand, 
by \eqref{roundish1}, the volume of an $i$-level $\fra$-deformed $\Tk^\fra$-super-tile is at least $O_\fra(1)\cdot \theta^{id}$. This implies the desired upper bound (ii).
\end{proof}


\subsection{Proof of Theorem \ref{thm1}}

\begin{proof} We start with (\ref{dim1}). The proof is essentially contained as a step in the proof of \cite[Theorem 1]{Tre20}.
We obtain upper bounds for the ergodic integrals $|S_R^{\Sk^\fra}(\phi,\lamb)|$, which are uniform in $\Sk^\fra \in \X^\fra_\omega$, which imply $L^2$-bounds, used in Lemma~\ref{lem-spec1}.
Split the integral over $Q_R$ into the sum over super-tiles of levels $i=0,\ldots,n$, using the decomposition from Lemma~\ref{lem-decomp} for $Q_R = [-R,R]^d$, plus the integral over the ``left-over'' part of $Q_R$, where the integral is estimated by $\|\phi\|_\infty$ times the volume, resulting in a $CR^{d-1}$ term (here and below $C,C',\ldots,$ are generic constants).

To be more precise,  let $\phi$ is a TLC function of level 0 of the form \eqref{TLC1}. By Lemma~\ref{lem-integ} and Lemma~\ref{lem-decomp}, for any $\Sk^\fra \in \X^\fra_\omega$,
\be \label{bound0}
\bigl|S_R^{\Sk^\fra}(\phi,\lamb)\bigr| \le CR^{d-1} + C \sum_{i=0}^n \sum_{j=1}^m \sum_{k=1}^{\kappa_j^{(i)}} \|\Mc(\bz,i)\vec\zeta\|,\ \ \mbox{where}\ \bz = {\bf L}_\fra^{\sf T}\lamb\ \ \mbox{and}\ \ 
\vec\zeta = \what\Psib(\lamb).
\ee
By the definition of the pointwise upper Lyapunov exponent, for any $\eps>0$ there exists $C_\eps>0$ such that 
\be \label{bound1}
\|\Mc(\bz,i)\vec\zeta\| \le C_\eps \exp\bigl[(\chi^+(\bz,\vec\zeta) + \eps)i\bigr],\ \ \mbox{for all}\ i\ge 0.
\ee
If $\chi^+(\bz,\vec\zeta) = d\log\theta$, the estimate \eqref{dim1} is trivial, so we we can assume $\chi^+(\bz,\vec\zeta) < d\log\theta$. Choose $\eps>0$ so that $\chi^+(\bz,\vec\zeta) +\eps < d\log\theta$, and if
$\chi^+(\bz,\vec\zeta)  < (d-1)\log\theta$, so that $\chi^+(\bz,\vec\zeta) +\eps < (d-1)\log\theta$,
By \eqref{bound0} and Lemma~\ref{lem-decomp}, applied to $Q_R = [-R,R]^d$, and writing $\chi^+ = \chi^+(\bz,\vec\zeta)$ for simplicity, we obtain
\begin{equation} \label{bound101}
\bigl|S_R^{\Sk^\fra}(\phi,\lamb)\bigr|   \le CR^{d-1} +  C R^{d-1} \sum_{i=0}^{n} \theta^{i-id} e^{(\chi^++\eps)i}.
\end{equation}
There are two cases. If $(d-1)\log\theta \le \chi^+  < d\log\theta$, then we continue
\begin{eqnarray*}
\bigl|S_R^{\Sk^\fra}(\phi,\lamb)\bigr|  & \le & CR^{d-1} + C' R^{d-1} \theta^{n(1-d)}e^{(\chi^++\eps)n}\\[1.2ex] & \le &  C''e^{(\chi^++\eps)n}  \\[1.1ex] & \le &  C'''R^{(\chi^++\eps)/\log\theta}, 
\end{eqnarray*}
where we used that  $\theta^n \asymp R$. If $\chi^+(\bz,\vec\zeta) +\eps < (d-1)\log\theta$, then \eqref{bound101} yields 
$$
\bigl|S_R^{\Sk^\fra}(\phi,\lamb)\bigr|   \le \wt{C}R^{d-1}.
$$
By Lemma~\ref{lem-spec1}, we obtain that $d^-(\sig_\phi,\lamb) \ge 2\min\{d - \frac{\chi^++\eps}{\log\theta},1\}$,
and  taking $\eps\to 0$ implies the desired inequality \eqref{dim1}.

\smallskip

{ In order to deduce \eqref{dim101} from \eqref{dim1}, it suffices to show that 
$$
\chi^+(0,\vec\zeta)=\log\vartheta_2,\ \ \mbox{where}\ \ \vec\zeta = \what \Psib(0),\ \Psib = (\psi_1,\ldots,\psi_m)^{\sf T},
$$
and $\phi(\Sk^\fra) = \sum_{k=1}^m \sum_{x\in \Lc_k(\Sk^\fra)} \delta_x * \psi_k(0)$ has mean zero on $(X_\om^\fra,\mu_\fra)$, where
$\supp(\psi_k) \subset \Int(T_k^\fra)$.
Since $\Cc(0,n) = (\Sf_\om^{\sf T})^n$, we have
$$
\chi^+(0,\vec\zeta) = \limsup_{n\to \infty} \frac{1}{n} \log \|(\Sf_\om^{\sf T})^n \cdot\what\Psib(0)\|.
$$
The claim will follow if we show that $\what\Psib(0)$ is orthogonal to the PF eigenvector of $\Sf_\om$. 
{ By Birkhoff's Ergodic Theorem and the definition of $\phi$, we have for a.e.\ $\Sk^\fra\in X_\om^\fra$:
\begin{eqnarray*}
\int_{X_\om^\fra} \phi \,d\mu_\fra & = & \lim_{R\to \infty} (2R)^{-d} \int_{Q_R} \phi(\Sk^\fra - \bt)\,d\bt\\
& = &\sum_{k=1}^m \Bigl(\int_{T_k^\fra} \psi_k\Bigr)\cdot \freq(T_k^\fra,\Sk^\fra) \\
& = & \sum_{k=1}^m \what\psi_k(0) \cdot \freq(T_k^\fra,\Sk^\fra),
\end{eqnarray*}
where we also used the existence of uniform patch frequencies.
}
Thus it suffices to check that
$\bigl(\freq(T_k^\fra,\Sk^\fra)\bigr)_{k\le m}$ is a PF eigenvector for $\Sf_\om$. But since uniform frequencies exist, for any $j\le m$,
$$
\freq(T_k^\fra,\Sk^\fra) = \frac{\#\{t\in \R^d: \ T^\fra_k + t \subset T_j^{\fra,n}\}}{\Lk^d(\supp(T_j^{\fra,n}))}=
\lim_{n\to \infty} \frac{\Sf_\om^n(k,j)}{\sum_{i=1}^m \Sf_\om^n(i,j)\cdot \Lk^d(T_i^\fra)}
$$
By the Perron-Frobenius Theorem, $\Sf_{\om}^n(i,j) \sim \vartheta_1^n r_i \ell_j$, where $(r_i)_{i\le m}$ and $(\ell_j)_{j\le m}$ are the (normalized) right and left PF eigenvectors of $\Sf_\om$
respectively. It follows that $\freq(T_k^\fra,\Sk^\fra) =  r_k$, as claimed.}
\end{proof}

\subsection{Proof of Proposition~\ref{prop-eigen}}

\begin{proof}
Every continuous eigenfunction is certainly measurable, so it suffices to show that if $\lamb\in \R^d$ is a ``measurable eigenvalue,''
then \eqref{eigen-cond1} holds.

Let $\bv$ be an elementary recurrence vector in $X_\om$. This implies, by definition \eqref{eq-recur1}, 
that there exist a tile $T_j\in \Tk$ and $\bx\in \R^d$ such that $T_j + \bx\in \Tk$, with $\bv = \alpha(z_1, z_2)$ for $z_1\in T_j$ and $\bx = z_2 - z_1$.

Then $T_j\cup (T_j+\bx) \subset \Tk$, and applying a
power of the tile substitution $\om^n$ we obtain
$$
\om^n(T_j) \cup (\om^n(T_j) + \varphi^n \bx)\subset \om^n(\Tk) \subset \Tk.
$$
Now fix a vertex $v$ of $\Tk$ and apply the deformation $\fra$ to the patch
$$
\bigl(\om^n(T_j)-v\bigr) \cup \bigl((\om^n(T_j) + \varphi^n \bx)-v\bigr)\subset  \Tk-v,
$$
which yields
$$
T_j^{\fra,n} \cup (T_j^{\fra,n} + \bx_n) \subset (\Tk-v)^\fra,
$$
for some $\bx_n\in \R^d$, see \eqref{def-prot2}. In order to compute $\bx_n$, note
that $(\varphi^n z_1, \varphi^n z_2)$ is a recurrence in $\Tk$, and the corresponding recurrence vector is given by $\alpha\varphi^n (z_1, z_2) =  M^n \bv$, hence
$$
\bx_n = {\bf L}_\fra M^n \bv,
$$
in view of \eqref{lin-map}. 

Let $\Pk$ be any $X_\om^\fra$-patch and $U\subset \R^d$ a Borel set of diameter less than $\eta(X^\fra_\om)$ (the diameter of the largest ball contained in every $\Tk^\fra$ prototile).
Consider the cylinder set 
$$
X^\fra_{\Pk,U}:= \Upsilon(\Pk) + U\subset X^\fra_\om.
$$

\begin{lemma} \label{lem-claim1}
There exists $\delta = \delta(\fra,\bx)>0$ 
such that 
\be \label{claim1}
\mu\bigl(X^\fra_{\Pk,U} \cap (X^\fra_{\Pk,U}+\bx_n)\bigr) \ge \delta \cdot\mu(X^\fra_{\Pk,U})\ \ \mbox{for all}\ n\ge n(\Pk).
\ee
\end{lemma}

Assuming \eqref{claim1} is verified, the proof of \eqref{eigen-cond1} proceeds exactly as in the proof of \cite[Theorem 4.3]{SolTil}. In fact, in \cite[Lemma 1.6]{SolTil} it is shown that any
finite local complexity and repetitive tiling space can be partitioned into a finite union of cylinder sets of arbitrarily small radius. Let $X^\fra_\om = \coprod_\ell X^\fra_{\Pk_\ell,V_\ell}$ be  such a
partition in our case. A measurable eigenfunction $f$ for $(\X^\fra_\omega,\R^d,\mu_\fra)$ (which can be assumed equal to one in modulus a.e.\ by ergodicity), corresponding to an eigenvalue $\lamb\in \R^d$, can be approximated uniformly on a set of full measure by a linear combination $g$ of characteristic functions of the cylinder sets $X^\fra_{\Pk_\ell,V_\ell}$. Suppose $\|f-g\|_\infty < \eps$, where
$\eps$ can be made arbitrarily small. Let
$$
A_{n,\eps}:= \coprod_\ell (X^\fra_{\Pk_\ell,V_\ell}\cap (X^\fra_{\Pk_\ell,V_\ell}+\bx_n)),
$$
and note that $\mu_\fra(A_{n,\eps}) \ge \delta$ by Lemma~\ref{lem-claim1}.
Thus,
$$
\Jk:=\int_{A_{n,\eps}} |f(\Sk^\fra - \bx_n) - f(\Sk^\fra)|\,d\mu_\fra = \Bigl|e^{-2\pi i \langle \lamb, \bx_n\rangle }-1\Bigr|\cdot \mu_\fra(A_{n,\eps})\ge \delta\cdot \Bigl|e^{-2\pi i \langle \lamb, \bx_n\rangle }-1\Bigr|,
$$
for $n\ge n_0 = n_0(\eps)$ by the eigenvalue equation. (Note that $n_0$ depends on the partition, which in turn depends on $\eps$.) On the other hand,
$$
\Jk \le 2\|f-g\|_\infty + \int_{A_{n,\eps}} |g(\Sk^\fra - \bx_n) - g(\Sk^\fra)|\,d\mu_\fra < 2\eps,
$$
since $g(\Sk^\fra - \bx_n) = g(\Sk^\fra)$ on $A_{n,\eps}$ by construction. It follows that
$$
\Bigl| e^{-2\pi i \langle \lamb, \bx_n\rangle}-1\Bigr| \le \frac{2\eps}{\delta},\ \ n\ge n_0(\eps),\ \ \mbox{where}\ \ \bx_n = {\bf L}_\fra M^n \bv,
$$
which is equivalent to \eqref{eigen-cond1}. Thus it remains to prove the lemma. \end{proof}

\begin{proof}[Proof of Lemma~\ref{lem-claim1}]
For $n$ sufficiently large, such that $\Pk$ does not intersect $\Pk+\bx_n$, we have
$$
X^\fra_{\Pk,U} \cap (X^\fra_{\Pk,U}+\bx_n)\supset X^\fra_{\Pk \cup (\Pk+\bx_n),U}.
$$
Hence by Proposition~\ref{prop-meas1}, it suffices to show that 
\be \label{freq-est}
\freq(\Pk \cup (\Pk+\bx_n),\Tk^\fra) \ge \delta\cdot \freq(\Pk,\Tk^\fra),
\ee
for some $\delta>0$, independent of $n\in \N$. For the undeformed, self-similar tiling $\Tk$, the analogous inequality is proved in \cite[Lemma 4.2]{SolTil}, and for the deformation
it follows by the quasi-isometry \eqref{eq1-qi}. 
More precisely, by repetitivity, there exists $k_0\in \N$ and a deformed super-tile $T_i^{\fra,k_0}$ such that 
$$
T_j^\fra \cup (T_j^\fra+\bx) \subset T_i^{\fra,k_0},
$$
where we replace the prototile $T_j^\fra$ by its translate, if necessary.
Then
$$
T_j^{\fra,n} \cup (T_j^{\fra,n}+\bx_n) \subset T_i^{\fra,n+k_0},\ \ n\ge 1.
$$
It follows that
$$
\#\bigl\{t:\ [\Pk \cup (\Pk+\bx_n)]+t \subset T_i^{\fra,n+k_0}\bigr\} \ge \#\bigl\{t:\ \Pk+t \subset T_j^{\fra,n}\bigr\},
$$
which implies (since uniform patch frequencies exist) that 
$$
\freq(\Pk \cup (\Pk+\bx_n),\Tk^\fra) \ge \delta\cdot \freq(\Pk,\Tk^\fra),
$$
where 
$$
\delta = \liminf_{n\to \infty} \frac{\Lk^d(T_j^{\fra,n})}{\Lk^d(T_i^{\fra,n+k_0})}\,.
$$
It remains to note that $\delta>0$ by the quasi-isometry claim \eqref{eq1-qi}.
\end{proof}

{\bf Acknowledgement.} We are grateful to Lorenzo Sadun for patiently explaining to us the subtleties of the AP-complex for PSS tilings and their deformations.
The images for Kenyon's tilings in Section 6 were constructed using Sage code developed by Mark Van Selous for the Laboratory of Experimental Mathematics at Maryland.

\end{document}